\documentclass[a4paper,11pt]{article}

\usepackage[T1]{fontenc}
\usepackage{lmodern}

\usepackage[latin1]{inputenc}
\usepackage{amsmath}
\usepackage{amsthm}
\usepackage{amssymb}
\usepackage{mathrsfs}
\usepackage{graphicx}
\usepackage[all]{xy}
\usepackage{hyperref}

\usepackage{stmaryrd} 

\newtheorem{thm}{Theorem}[section]
\newtheorem{cor}[thm]{Corollary}
\newtheorem{claim}[thm]{Claim}

\newtheorem{lemma}[thm]{Lemma}
\newtheorem{prop}[thm]{Proposition}
\newtheorem{theo}[thm]{Theorem}

\theoremstyle{definition}

\newtheorem{remark}[thm]{Remark}

\newtheorem{de}[thm]{Definition}

\newcommand{\BBT}{\mathrm{BBT}}
\newcommand{\Lip}{\mathrm{Lip}}
\newcommand{\vol}{\overline{\mathrm{vol}}}
\newcommand{\Aut}{\mathrm{Aut}}
\newcommand{\Out}{\mathrm{Out}}
\newcommand{\Fix}{\mathrm{Fix}}
\newcommand{\Axis}{\mathrm{Axis}}

\newcommand{\Stab}{\mathrm{Stab}}
\newcommand{\ad}{\mathrm{ad}}
\newcommand{\id}{\mathrm{Id}}
\newcommand{\Char}{\mathrm{Char}}
\newcommand{\Isom}{\mathrm{Isom}}

\newcommand{\cald}{\mathscr{D}}
\newcommand{\calp}{\mathcal{P}}
\newcommand{\calz}{\mathcal{Z}}

\title{Acylindrical hyperbolicity of automorphism groups of infinitely-ended groups}
\date{\today}
\author{Anthony Genevois and Camille Horbez}

\begin{document}

\maketitle

\begin{abstract}
We prove that the automorphism group of every infinitely-ended finitely generated group is acylindrically hyperbolic. In particular $\Aut(\mathbb{F}_n)$ is acylindrically hyperbolic for every $n\ge 2$. More generally, if $G$ is a group which is not virtually cyclic, and hyperbolic relative to a finite collection $\mathcal{P}$ of finitely generated proper subgroups, then $\Aut(G,\calp)$ is acylindrically hyperbolic. 

As a consequence, a free-by-cyclic group $\mathbb{F}_n\rtimes_{\varphi}\mathbb{Z}$ is acylindrically hyperbolic if and only if $\varphi$ has infinite order in $\Out(\mathbb{F}_n)$.
\end{abstract}

\tableofcontents

\newpage

\section{Introduction}

\noindent
Following the seminal work of Gromov on hyperbolic groups \cite{Gro}, it is a natural problem in geometric group theory to look for aspects of negative curvature among groups. In this vein, Osin defined a group $G$ to be \emph{acylindrically hyperbolic} if it admits a nonelementary acylindrical action on a hyperbolic space \cite{OsinAcyl}. As follows from work of Dahmani, Guirardel and Osin \cite{DGO}, acylindrical hyperbolicity turns out to have important algebraic consequences for the group: most strikingly, every acylindrically hyperbolic group $G$ has uncountably many normal subgroups, and is \emph{SQ-universal}, i.e.\ every countable group embeds in a quotient of $G$.

\medskip\noindent
One of the most natural questions about a given group $G$ is to understand its automorphism group $\Aut(G)$. In \cite{AutRAAG,MR4011668}, the first named author investigated the following question: under which conditions can negative curvature for $G$ be promoted to negative curvature for $\Aut(G)$? In a sense, one wants to understand how $\mathrm{Inn}(G)$ sits inside $\Aut(G)$ to deduce negative curvature properties of $\Aut(G)$. In \cite{MR4011668}, it was proved that the automorphism group of every one-ended hyperbolic group is acylindrically hyperbolic. Also, in \cite{AutRAAG}, the first named author characterised the acylindrical hyperbolicity of automorphism groups of one-ended right-angled Artin and right-angled Coxeter groups; similar results about the larger family of graph products of groups can be found in \cite{AutGP, AutRAAG}. In the present paper, we investigate the geometry of $\Aut(G)$ when $G$ is infinitely-ended. Our main result is the following.

\begin{theo}\label{thm:MainIntro}
Let $G$ be a finitely generated infinitely-ended group. Then $\Aut(G)$ is acylindrically hyperbolic.
\end{theo}

\noindent
In particular, denoting by $\mathbb{F}_n$ a free group of rank $n$, we get the following consequence.

\begin{cor}\label{cor:IntroFree}
For every $n \geq 2$, the group $\mathrm{Aut}(\mathbb{F}_n)$ is acylindrically hyperbolic.
\end{cor}

\noindent
We would like to make a few comments on these statements. First, in Theorem~\ref{thm:MainIntro}, the assumption that $G$ be finitely generated is important. For instance, the automorphism group $\Aut(\mathbb{F}_\omega)$ of a free group of countably infinite rank fails to be acylindrically hyperbolic. It is not even enough to assume that $G$ splits as a free product of finitely many groups that are either freely indecomposable or isomorphic to $\mathbb{Z}$, see Remark~\ref{rk:finite-generation}.   

\medskip \noindent
Second, acylindrical hyperbolicity of $\mathrm{Out}(\mathbb{F}_n)$ follows from work of Bestvina and Feighn \cite{BF}. But the acylindrical hyperbolicity of $\mathrm{Aut}(\mathbb{F}_n)$ and of $\mathrm{Out}(\mathbb{F}_n)$ do not seem to follow from one another. Our approach is in fact different: the generalised loxodromic elements in our case are inner automorphisms associated to sufficiently filling elements of $\mathbb{F}_n$ -- elements that are not elliptic in any small $\mathbb{F}_n$-action on a real tree; in the case of $\mathrm{Out}(\mathbb{F}_n)$, Bestvina and Feighn proved that fully irreducible outer automorphisms are generalised loxodromic elements. 

\medskip\noindent
From this point of view, we would like to emphasize that the situation is more complicated here in the case of free groups than in case of a closed surface group $\pi_1(\Sigma_g)$ with $g\ge 2$. In the surface setting, passing from $\Out(\pi_1(\Sigma_g))$ to $\Aut(\pi_1(\Sigma_g))$ amounts to considering the mapping class group of a once-punctured surface, which still acts acylindrically on the associated curve graph; in the free group case, to our knowledge, our construction yields the first example of a nonelementary isometric action of $\Aut(\mathbb{F}_n)$ that does not factor through $\mathrm{Out}(\mathbb{F}_n)$ -- some analogues of `punctured curved graphs' fail to be hyperbolic, see \cite{Ham}.

\medskip\noindent 
Notice also that the obvious analogue of Theorem~\ref{thm:MainIntro} for $\Out(G)$ is not true as such: for example, if $G$ splits as $G=A\ast B$, where $A$ and $B$ are freely indecomposable infinite groups with trivial outer automorphism group and trivial center, then $\Out(G)$ is virtually isomorphic to $A\times B$ (as follows from work of Levitt \cite{Lev}), which is not acylindrically hyperbolic. Whether or not $\Out(G)$ is acylindrically hyperbolic when $G$ is finitely generated and splits as a free product with at least three factors is an open question to our knowledge.

\medskip\noindent
Combined with \cite{MR4011668}, it follows from Theorem \ref{thm:MainIntro} that the automorphism group of any nonelementary hyperbolic group is acylindrically hyperbolic. Actually, the arguments from \cite{MR4011668} can be extended to one-ended relatively hyperbolic groups (alternatively, see Proposition~\ref{prop:RH} below). In the case of a group $G$ which is hyperbolic relative to a finite collection $\mathcal{P}$ of finitely generated subgroups, a natural subgroup  of $\Aut(G)$ to consider is the group $\Aut(G,\mathcal{P})$ made of all automorphisms sending every subgroup in $\mathcal{P}$ to a conjugate of a subgroup in $\mathcal{P}$. We establish the following theorem, which refines Theorem~\ref{thm:MainIntro}. 

\begin{thm}\label{thm:IntroRHacyl}
Let $G$ be a group which is hyperbolic relative to a finite collection $\mathcal{P}$ of finitely generated subgroups. Assume that the pair $(G, \mathcal{P})$ is nonelementary (i.e.\ $G$ is not virtually cyclic, and all subgroups in $\mathcal{P}$ are proper). Then $\Aut(G,\mathcal{P})$ is acylindrically hyperbolic.

\noindent If in addition no subgroup in $\mathcal{P}$ is relatively hyperbolic, then $\Aut(G)$ is acylindrically hyperbolic.
\end{thm}

\noindent The additional part of the statement follows from the fact that when no subgroup in $\mathcal{P}$ is relatively hyperbolic, then $\Aut(G)=\Aut(G,\calp')$, where $\calp'$ is obtained from $\calp$ by removing all finite groups.

\medskip\noindent
Theorem~\ref{thm:IntroRHacyl} covers a large class of interesting groups. Besides hyperbolic groups, examples of relatively hyperbolic groups include fundamental groups of complete Riemannian manifolds of finite volume with pinched negative sectional curvature \cite{MR1650094}, some mapping tori of surface groups \cite[Proposition~6.2]{MR2262721}, \cite[Theorem~4.9]{MR2421140}, \cite{MR3572767}, fundamental groups of closed 3-manifolds which are not graph manifolds \cite[Corollary~E]{MR3079269}, toral relatively hyperbolic groups such as limit groups \cite{MR2131400, Dah}, free-by-cyclic groups $\mathbb{F}_n \rtimes_\varphi \mathbb{Z}$ where $n \geq 2$ and $\varphi$ has exponential growth (see the discussion after the statement of Corollary~\ref{cor:Intro-free-by-cyclic}), small cancellation quotients \cite{MR1706019, MR3353035, GruberThesis}, some graph products of groups \cite{Qm} (including some right-angled Coxeter groups \cite{MR3623669, ConingOff}), some graph braid groups \cite{GraphBraid}. 

\medskip \noindent
As an application of our work, we obtain information on the geometry of free-by-cyclic groups. More generally, combining Theorem~\ref{thm:MainIntro} with \cite[Theorem~1.5]{MR4011668}, we get the following statement.

\begin{cor}\label{cor:IntroExtensionFree}
Let $G$ be a finitely generated infinitely-ended group, let $H$ be a group and let $\varphi : H \to \mathrm{Aut}(G)$ be a homomorphism. The semidirect product $G \rtimes_\varphi H$ is acylindrically hyperbolic if and only if
$$\mathrm{ker} \left( H \overset{\varphi}{\to} \mathrm{Aut}(G) \to \mathrm{Out}(G) \right)$$
is a finite subgroup of $H$.
\end{cor}

\noindent
In the particular case where $G$ is free and $H$ is infinite cyclic, this can be reformulated as follows.

\begin{cor}\label{cor:Intro-free-by-cyclic}
Let $n\ge 2$, and let $\varphi\in\Aut(\mathbb{F}_n)$. Then $\mathbb{F}_n \rtimes_\varphi \mathbb{Z}$ is acylindrically hyperbolic if and only if the image of $\varphi$ in $\Out(\mathbb{F}_n)$ has infinite order.
\end{cor}

\noindent Corollary~\ref{cor:Intro-free-by-cyclic} fits in a long sequence of results regarding the geometry of free-by-cyclic groups. A theorem of Brinkmann \cite{Bri} based on work of Bestvina and Feighn \cite{BF-comb} asserts that $\mathbb{F}_n\rtimes_\varphi\mathbb{Z}$ is hyperbolic if and only if $\varphi$ is atoroidal (i.e.\ no nontrivial power of $\varphi$ preserves a conjugacy class in $\mathbb{F}_n$). Also $\mathbb{F}_n\rtimes_\varphi\mathbb{Z}$ is relatively hyperbolic if and only if $\varphi$ has exponential growth \cite{GauteroLustig,Gho,DL,Mac,Hag}, and in this case it is hyperbolic relative to the mapping torus of the collection of maximal polynomially growing subgroups of $\mathbb{F}_n$ for the outer class of $\varphi$. Finally, it was proved in \cite[Corollary 4.3]{MR3548122} that $\mathbb{F}_n\rtimes_\varphi\mathbb{Z}$ is virtually acylindrically hyperbolic if and only if the image of $\varphi$ in $\Out(\mathbb{F}_n)$ has infinite order. But to our knowledge, it is currently unknown whether virtually acylindrically hyperbolic groups are acylindrically hyperbolic (see \cite{MR3968890}), so Corollary~\ref{cor:Intro-free-by-cyclic} is new to our knowledge.

\paragraph*{A word on the proof.} Our proof of Theorem~\ref{thm:MainIntro} relies on a celebrated construction of \emph{projection complexes} due to Bestvina, Bromberg and Fujiwara \cite{BBF}. This requires having a collection $\mathbb{Y}$ of metric spaces on which $\Aut(G)$ acts (preserving the metrics), together with projection maps between these spaces satisfying certain axioms -- most importantly, a version of the Behrstock inequality from \cite{Beh}, and the fact that there are only finitely many large projections between any two distinct spaces $X,Z\in\mathbb{Y}$. From that, Bestvina, Bromberg and Fujiwara build an action of $\Aut(G)$ on a \emph{projection complex} $\mathcal{P}$ which is a quasi-tree, and under a good control on $\Aut(G)$-stabilisers of collections of spaces in $\mathbb{Y}$ one can deduce that the $\Aut(G)$-action on $\mathcal{P}$ is acylindrical \cite{BBFS}.

\medskip\noindent Let us now describe the collection $\mathbb{Y}$ of spaces we work with in the case where $G=\mathbb{F}_n$ with $n\ge 2$. The group $\Aut(\mathbb{F}_n)$ acts on the (unprojectivized) \emph{Auter space}: this is the space of $\mathbb{F}_n$-equivariant isometry classes of basepointed free simplicial minimal actions of $\mathbb{F}_n$ on simplicial metric trees, and $\Aut(\mathbb{F}_n)$ acts by precomposition of the action. Let $g\in G$ be a generic element: precisely, we require that $g$ is not elliptic in any real $\mathbb{F}_n$-tree with cyclic arc stabilisers -- those contain all trees that appear in the Culler--Morgan compactification of Outer space. We fix a Cayley tree $S_0$ of $\mathbb{F}_n$, and let $A_g$ be the subspace of Auter space made of all basepointed trees $(S_0,p)$ with $p$ varying along the axis of $g$ in $S_0$. Now our collection $\mathbb{Y}$ consists of the $\Aut(\mathbb{F}_n)$-orbit of $A_g$, a collection of subspaces of Auter space. We then define a notion of projections between these various subspaces. Understanding these projections amounts to understanding the closest-point projections within the fixed Cayley tree $S_0$ between the axes of the elements $\varphi(g)$ with $\varphi$ varying in $\Aut(G)$, and we check the Bestvina--Bromberg--Fujiwara axioms there. A technical tool that turns out to be crucial in this analysis is a \emph{persistence of long intersections} property, which essentially amounts to saying that if an element $h\in\mathbb{F}_n$ contains a high power of $g$ as a subword, then for every $\varphi\in\Aut(\mathbb{F}_n)$, the element $\varphi(h)$ contains a high power of $\varphi(g)$. Geometrically, if the axes of $g$ and $h$ have a long intersection in a Cayley tree of $\mathbb{F}_n$, then for every automorphism $\varphi\in\Aut(\mathbb{F}_n)$, the same is true for $\varphi(g)$ and $\varphi(h)$. Here it is crucial that the element $g$ be sufficiently generic: if $g$ were a basis element, one could undo a high power of $g$ by choosing for $\varphi$ a high power of a Dehn twist automorphism. The proof of this property for elements $g$ as above uses a limiting argument where we show that if the property fails, the element $g$ would have to become elliptic in a tree in the boundary of Outer space.

\paragraph*{Organization of the paper.} In Section~\ref{sec:prelim}, we review basic notions concerning group actions on trees, as well as the Bestvina--Bromberg--Fujiwara construction. We also establish some easy consequences of having a WPD action on a tree that are important in the sequel. In Section~\ref{sec:criterion}, we give a general abstract criterion that ensures that $\Aut(G)$ is acylindrically hyperbolic, from a $G$-action on a simplicial tree $T$ with a control on overlaps between axes of sufficiently generic elements. In Section~\ref{sec:rh-one-ended}, we apply this criterion to a JSJ splitting to establish Theorem~\ref{thm:IntroRHacyl} when $G$ is a one-ended relatively hyperbolic group. In Section~\ref{sec:infinitely-ended}, we prove our main theorem showing the acylindrical hyperbolicity of automorphism groups of finitely generated infinitely-ended groups. Finally, in Section~\ref{sec:rh}, we combine the results from the previous two sections to prove Theorem~\ref{thm:IntroRHacyl} when $G$ is any relatively hyperbolic group, without any assumption on the number of ends.  

\paragraph*{Acknowledgments.} We are very grateful to the anonymous referee for many comments and suggestions that significantly improved the exposition of the paper. The first named author was supported by a public grant as part of the Fondation Math\'ematique Jacques Hadamard. The second named author acknowledges support from the Agence Nationale de la Recherche under Grant ANR-16-CE40-0006 DAGGER.

\section{Preliminaries}\label{sec:prelim}

\subsection{Group actions on trees}

\noindent Let $G$ be a group. A \emph{$G$-tree} is a real tree $T$ equipped with an isometric $G$-action. Given a $G$-tree $T$ and an element $g\in G$, either $g$ fixes a point in $T$ (in which case $g$ is said to be \emph{elliptic} in $T$), or else there is a (unique) subspace $\ell$ of $T$ that is homeomorphic to a line and left invariant by $g$, and $g$ acts on $\ell$ by translation (in this case $g$ is said to be \emph{loxodromic} in $T$, and $\ell$ is called the \emph{axis} of $g$ in $T$, denoted by $\mathrm{Axis}_T(g)$). The \emph{characteristic subset} of $g$, denoted by $\mathrm{Char}_T(g)$, is either its fixed point set $\Fix_T(g)$ if $g$ is elliptic, or its axis if $g$ is loxodromic. The \emph{translation length} of $g$ in $T$ is defined as $||g||_T=\inf_{x\in T}d(x,gx)$. This infimum is always achieved: one has $||g||_T=0$ if and only if $g$ is elliptic, and otherwise $||g||_T$ is also equal to the amount of translation of $g$ on its axis. 

\medskip\noindent A $G$-tree $T$ is \emph{nonelementary} if $G$ contains two elements acting loxodromically on $T$ and whose axes in $T$ intersect in a compact (possibly empty) subset. A $G$-tree $T$ is \emph{minimal} if it does not contain any proper nonempty $G$-invariant subtree. As established in the proof of \cite[Proposition~3.1]{CM}, a nonelementary minimal $G$-tree is always equal to the union of all axes of loxodromic elements. 

\medskip\noindent Throughout the paper, all simplicial trees are equipped with the simplicial metric where all edges are assigned length $1$. 

\medskip\noindent We now record some general lemmas about group actions on trees that will be useful in the paper.

\begin{lemma}\label{lem:center}
Let $G$ be a group, and let $T$ be a $G$-tree. Assume that the $G$-action on $T$ is minimal and nonelementary. Let $\iota\in\Isom(T)$ be an isometry which commutes with the $G$-action, i.e.\ for every $x\in T$ and every $g\in G$, one has $\iota(gx)=g\iota(x)$. Then $\iota=\mathrm{id}$.

\medskip\noindent In particular, the center $Z(G)$ of $G$ lies in the kernel of the $G$-action on $T$.
\end{lemma}

\begin{proof}
The last assertion of the lemma follows from the first, since every element of $Z(G)$ determines an isometry of $T$ that commutes with the $G$-action. We thus focus on proving the first part of the lemma.

\medskip\noindent We claim that $\iota$ fixes the axis of every element $h\in G$ acting loxodromically on $T$ pointwise. Since the $G$-action on $T$ is minimal, as recalled above, the tree $T$ is equal to the union of all axes of loxodromic elements, so this claim will be enough to conclude. 

\medskip\noindent We now prove the above claim. Observe that $\iota$ sends the axis of every loxodromic element $h \in G$ to the axis of $\iota h \iota^{-1}=h$. Therefore, $\iota$ stabilises the axis of $h$. Moreover, because a translation and a reflection on a line do not commute, we know that $\iota$ acts on the axis of $h$ as a translation. Next, as the $G$-action on $T$ is nonelementary, for every element $h\in G$ acting loxodromically on $T$, we can find an element $h'\in G$ which acts loxodromically on $T$ and whose axis is disjoint from the axis of $h$ \cite[Lemma~2.1]{CM}. Then $g$ fixes the closest-point projection of the axis of $h'$ onto the axis of $h$ (which consists of a single point). As $\iota$ acts by translation on $\Axis_T(h)$, this implies that $\iota$ fixes $\Axis_T(h)$ pointwise, thus proving our claim.
\end{proof}

\noindent 
A subset of a tree $T$ is said to be \emph{nondegenerate} if it contains more than one point, and \emph{degenerate} otherwise. Given two subtrees $I,J\subseteq T$ with degenerate intersection, the \emph{bridge} between $I$ and $J$ is the shortest closed segment of $T$ (possibly reduced to a single point) that intersects both $I$ and $J$. 

\begin{lemma}[{Paulin \cite[Proposition~1.6(1)]{Pau}}]\label{lemma:culler-morgan}
Let $G$ be a group, and let $T$ be a $G$-tree. Let $g,h\in G$ be two elements whose characteristic sets in $T$ are disjoint. Then $gh$ is loxodromic in $T$, and the bridge between $\Char_T(g)$ and $\Char_T(h)$ is contained in $\Axis_T(gh)$.
\end{lemma}

\begin{lemma}\label{lemma:action-on-tree}
Let $G$ be a group, and let $T$ be a $G$-tree. Let $g,h,h'\in G$ be elements that are loxodromic in $T$. If the distance between the projections of $\Axis_T(h)$ and of $\Axis_T(h')$ onto $\Axis_T(g)$ is greater than $\|g\|_T$, then the characteristic sets of $gh$ and of $h'$ are disjoint.  
\end{lemma}

\begin{proof}
Let $z$ be the point in the projection of $\Axis_T(h)$ onto $\Axis_T(g)$ which is the closest to $\Axis_T(h')$.  

\medskip \noindent
We claim that for every $x\in\Axis_T(h')$, one has $d(x,ghx) > d(z, ghz)$. Indeed, notice that $z$ belongs to the geodesic $[x,hx]$. As a consequence, $gz$ belongs to the geodesic $[gh x, x]$: this is clear if $g$ translates $z$ further away from $\Axis_T(h')$, and otherwise this follows from our assumption that the distance between the projections of $\Axis_T(h)$ and of $\Axis_T(h')$ onto $\Axis_T(g)$ is greater than $\|g\|_T$. Therefore,
$$d(x,ghx)= d(x,gz)+d(gz,ghx) \geq d(z,hx).$$
Because the concatenation of geodesics $[x,z] \cup [z,hz] \cup [hz,hx]$ is again a geodesic, we also have
$$d(z,hx) = d(z,hz) + d(z,x) > \|h\|_T + 2 d(z,\Axis_T(h)) + \|g\|_T.$$
But $d(z,ghz) \leq d(z,gz)+ d(z,hz) = \|g\|_T+ \|h\|_T+2d(z,\Axis_T(h))$, so $d(x,ghx) > d(z, ghz)$ as desired. 

\medskip \noindent
The above claim implies that no element in $\Axis_T(h')$ is minimally displaced by $gh$, in other words $\Axis_T(h') \cap \Char_T(gh) = \emptyset$.
\end{proof}

\noindent Given a tree $T$ and a point $x\in T$, a \emph{direction} at $x$ in $T$ is a connected component of $T\setminus\{x\}$.

\begin{lemma}\label{lemma:directions}
Let $G$ be a group, and let $T$ be a $G$-tree. Let $x\in T$, let $d\subseteq T$ be a direction at $x$, and let $g\in G$.

If $gd\subsetneq d$, then $g$ is loxodromic in $T$ and $[x,gx]\subseteq\Axis_T(g)$.
\end{lemma}

\begin{proof}
Assume towards a contradiction that $g$ is elliptic in $T$. If $gx=x$, then $d$ and $gd$ are two directions at $x$, so they are either equal or disjoint; in particular we cannot have $gd\subsetneq d$. If $gx\neq x$, let $m$ be the midpoint of $[x,gx]$. Then $gm=m$. If $d$ is the direction at $x$ that contains $m$, then $gd$ is the direction at $gx$ that contains $m$, and in this case $gd$ is not contained in $d$. Otherwise, the direction $d$ is contained in the unique direction at $m$ that contains $x$, and $gd$ is contained in the unique direction at $m$ that contains $gx$, and these directions are disjoint, so again we cannot have $gd\subsetneq d$.

This proves that $g$ is loxodromic in $T$, and we are left proving that $x\in\Axis_T(g)$. Otherwise, let $y$ be the closest-point projection of $x$ to $\Axis_T(g)$. If $d$ is the direction at $x$ that contains $y$, then $gd$ is the direction at $gx$ that contains $gy$, and $gd$ is not contained in $d$. Otherwise $d$ is contained in the direction at $y$ that contains $x$, and $gd$ is contained in the direction at $gy$ that contains $gx$, and $d\cap gd=\emptyset$. This concludes our proof.      
\end{proof}

\begin{lemma}\label{lemma:tree-extra}
Let $G$ be a group, and let $T$ be a $G$-tree. Let $g,h\in G$ be $T$-loxodromic elements, such that $\Axis_T(g)$ and $\Axis_T(h)$ intersect in a compact (possibly empty) segment. Then $$\bigcap_{n,m\in\mathbb{Z}\setminus\{0\}}\Char_T(g^nh^m)$$ is empty if $\Axis_T(g)\cap\Axis_T(h)$ is nondegenerate, and otherwise it is equal to the bridge between $\Axis_T(g)$ and $\Axis_T(h)$.   
\end{lemma}
\begin{figure}
\begin{center}
\includegraphics[scale=0.4]{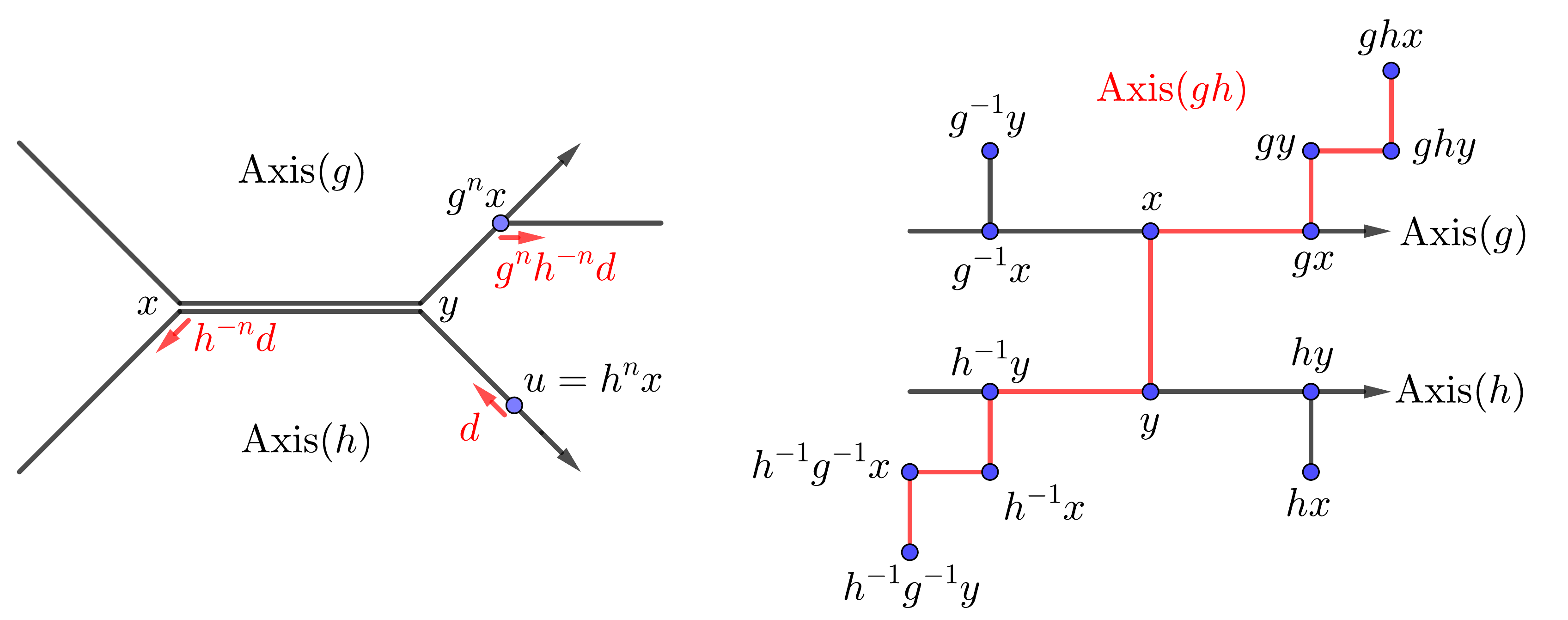}
\caption{The two cases considered in the proof of Lemma \ref{lemma:tree-extra}.}
\label{Tree}
\end{center}
\end{figure}

\begin{proof}
We refer to Figure \ref{Tree} for an illustration of the configurations of points we consider below.

\medskip \noindent
We first assume that $\Axis_T(g)\cap\Axis_T(h)$ is a nondegenerate segment $I=[x,y]$, and we let $D$ be its length. Up to replacing $h$ by $h^{-1}$, we will assume that $g$ and $h$ both translate in the direction from $x$ to $y$ along $I$. Let $n\in\mathbb{N}$ be such that $\|g^n\|_T > D$ and $\|h^n\|_T>D$. We will prove that $g^nh^{-n}$ and $g^{-n}h^n$ are $T$-loxodromic and that their axes do not intersect.

\medskip\noindent
Let $u=h^nx$, and let $d$ be the direction at $u$ that contains $x$. Then the direction $g^nh^{-n}d$ is contained in $d$, and the intersection between $[u,g^nh^{-n}u]$ and $I$ is reduced to $\{y\}$. By Lemma~\ref{lemma:directions}, this implies that $g^nh^{-n}$ is $T$-loxodromic and that its axis contains $[u,g^nh^{-n}u]$, and therefore $\Axis_T(g^nh^{-n})\cap I=\{y\}$. Likewise, one shows that $\Axis_T(g^{-n}h^n)\cap I=\{x\}$. This implies that $\Axis_T(g^nh^{-n})\cap\Axis_T(g^{-n}h^n)=\emptyset$, as desired.

\medskip\noindent
We now assume that $\Axis_T(g)\cap\Axis_T(h)$ is degenerate, and let $I=[x,y]$ be the bridge between $\Axis_T(g)$ and $\Axis_T(h)$, with $x\in\Axis_T(g)$ and $y\in\Axis_T(h)$ (possibly $x=y$ if $\Axis_T(g)\cap\Axis_T(h)$ is reduced to one point). By Lemma~\ref{lemma:culler-morgan}, for every $n,m\in\mathbb{Z}\setminus\{0\}$, the element $g^nh^m$ is $T$-loxodromic, and $I$ is contained in $\Axis_T(g^nh^m)$. In addition, we observe that $\Axis_T(gh)\cap\Axis_T(g)=[x,gx]$ and $\Axis_T(gh)\cap\Axis_T(h)=[y,h^{-1}y]$, while $\Axis_T(g^{-1}h^{-1})\cap\Axis_T(g)=[x,g^{-1}x]$ and $\Axis_T(g^{-1}h^{-1})\cap\Axis_T(h)=[y,hy]$. This implies that $\Axis_T(gh)\cap\Axis_T(g^{-1}h^{-1})=I$, which concludes our proof.
\end{proof}

\subsection{Review of the Bestvina--Bromberg--Fujiwara construction}

\noindent 
Recall that given a group $G$ acting on a hyperbolic metric space $Y$, an element $g\in G$ is \emph{loxodromic} in $Y$ if for some (equivalently, any) $y\in Y$, the orbit map $n\mapsto g^ny$ is a quasi-isometric embedding of $\mathbb{Z}$ into $Y$.

\medskip\noindent 
An action of a group $G$ on a metric space $Y$ is \emph{acylindrical} if for every $R>0$, there exist $L,N>0$ such that given any two points $x,y\in Y$ with $d(x,y)>L$, there are at most $N$ elements $g\in G$ such that $d(x,gx)<R$ and $d(y,gy)<R$. A group $G$ is \emph{acylindrically hyperbolic} \cite{OsinAcyl} if $G$ admits a nonelementary acylindrical action on a hyperbolic space. An element $g \in G$ is a \emph{generalised loxodromic element} if it is loxodromic with respect to some acylindrical action of $G$ on a hyperbolic space.

\medskip \noindent The proof of our main results is based on a celebrated construction of Bestvina, Bromberg and Fujiwara \cite{BBF}, further developed by the same authors and Sisto in \cite{BBFS}, which gives a criterion for proving that a group is acylindrically hyperbolic.

\medskip\noindent
Let $\mathbb{Y}$ be a collection of metric spaces, and let $G$ be a group. A $G$-action on $\mathbb{Y}$ is \emph{metric-preserving} if for every $g\in G$ and $Y\in\mathbb{Y}$, there exists an isometry $\iota_g^Y:Y\to g\cdot Y$, so that for every $g,h\in G$ and every $Y\in\mathbb{Y}$, one has $\iota_g^{h\cdot Y}\circ\iota_h^Y=\iota_{gh}^Y$.

\begin{thm}[Bestvina--Bromberg--Fujiwara--Sisto \cite{BBFS}]\label{thm:BBF}
Let $\delta\ge 0$, let $\mathbb{Y}$ be a collection of $\delta$-hyperbolic geodesic metric spaces, and let $G$ be a group acting in a metric-preserving way on $\mathbb{Y}$. Assume that there exists an assignment of a subset $\pi_Y(Z)\subseteq Y$ for any two distinct $Y,Z\in\mathbb{Y}$, such that if we set $$d_Y(X,Z) : = \mathrm{diam}_Y(\pi_Y(X) \cup \pi_Y(Z))$$ for all pairwise distinct $X,Y,Z \in \mathbb{Y}$, the following conditions hold:
 \begin{enumerate}
 \item\textbf{Projection axioms.} There exists $\theta \geq 0$ such that for all pairwise distinct $X,Y,Z,W \in \mathbb{Y}$, the following conditions hold:
\begin{itemize}
	\item[(P0)] $d_Y(X,X) \leq \theta$;
	\item[(P1)] if $d_Y(X,Z)> \theta$ then $d_X(Y,Z) \leq \theta$;
	\item[(P2)] $\{U \neq X,Z \mid d_U(X,Z) > \theta\}$ is finite;
\end{itemize}
\item\textbf{Unboundedness.} There exists at least one $Y \in \mathbb{Y}$ such that $\mathrm{Stab}_G(Y)$ has unbounded orbits in $Y$.
\item\textbf{Equivariance of $\pi$.} For every $g\in G$ and for all distinct $X,Y \in\mathbb{Y}$, one has $\pi_{gX}(gY)=\iota_g^X(\pi_X(Y))$. 
\item\textbf{Local acylindricity.} For every $Y \in \mathbb{Y}$, $\mathrm{Stab}_G(Y)$ acts acylindrically on $Y$ with uniform constants independent of $Y$, i.e. for every $R>0$, there exist $L,N>0$ such that, for every $Y \in \mathbb{Y}$ and for all $x,y\in Y$ with $d(x,y)>L$, there are at most $N$ elements $h\in \mathrm{Stab}_G(Y)$ such that $d(x,hx)<R$ and $d(y,hy)<R$. 
\item\textbf{Global acylindricity.} There exist $N,B \geq 1$ such that the pointwise stabiliser in $G$ of every subset of $\mathbb{Y}$ of cardinality at least $N$ has cardinality at most $B$. 
\end{enumerate}
Then $G$ acts acylindrically on a hyperbolic space $X$. Moreover, for every $Y \in \mathbb{Y}$, every $Y$-loxodromic element of $\mathrm{Stab}_G(Y)$ is $X$-loxodromic. As a consequence, $G$ is either virtually cyclic or acylindrically hyperbolic.
\end{thm}

\begin{proof}
For any two distinct $Y,Z \in \mathbb{Y}$, define a new subset $\pi_Y'(Z) \subset Y$ as in \cite[Theorem 4.1]{BBFS}. Also, for all pairwise distinct $X,Y,Z \in \mathbb{Y}$, set
$$\delta_Y(X,Z):= \mathrm{diam}_Y (\pi_Y'(X) \cup \pi_Y'(Z)).$$
According to \cite[Theorem 4.1]{BBFS}:
\begin{itemize}
	\item our new functions $\delta_Y$ satisfy the \emph{stronger projection axioms} $(SP1)-(SP5)$ from \cite[Section 2]{BBFS}, and in particular the projection axioms $(P0)-(P2)$, for $11 \theta$;
	\item for every $Y \in \mathbb{Y}$, one has $|d_Y-\delta_Y| \leq 2 \theta$;
	\item the projection map $\pi'$ is $G$ -equivariant, i.e. the equality $\pi_{gX}'(gY)= \iota_g^X (\pi_X'(Y))$ holds for every $g \in G$ and for all distinct $X,Y \in \mathbb{Y}$.
\end{itemize}
Given $K>0$, say that $X,Y\in\mathbb{Y}$ have no $K$-large projection if for every $Z\in\mathbb{Y}\setminus\{X,Y\}$, one has $\delta_Z(X,Y)\le K$. Define $\mathcal{C}_K(\mathbb{Y})$ to be the space obtained from the union of the metric spaces in $\mathbb{Y}$ by adding an edge of length $K$ between every point $x\in\pi_Y (X)$ and every point $y\in \pi_X(Y)$ whenever $X$ and $Y$ have no $K$-large projection. Then, if $K$ is sufficiently large compared to $\theta$ (see \cite[Lemma~3.1]{BBFS}), $\mathcal{C}_K(\mathbb{Y})$ is naturally endowed with a metric and $G$ acts on it isometrically. According to \cite[Theorem~6.4 and Corollary~6.8]{BBFS}, if $K\geq 4\theta$ then $\mathcal{C}_K(\mathbb{Y})$ is hyperbolic and $G$ acts acylindrically on it. 

\medskip \noindent
Observe that, for every $Y \in \mathbb{Y}$, the inclusion $Y \subset \mathcal{C}_K(\mathbb{Y})$ induces a quasi-isometric embedding. Indeed, recall that all maps $\delta_Z$ satisfy the projection axiom $(P0)$ for $11\theta$. Let $K>11\theta$. Fixing two points $x,z \in Y$, it follows from \cite[Theorem 6.3]{BBFS} that 
$$\frac{1}{4} \rho_Y(x,z) \leq d_\mathcal{C}(x,z) \leq 2 \rho_Y(x,z) +3K,$$
where $\rho_Y$ denotes the metric of $Y$. 
As the inclusion $Y \subset \mathcal{C}_K(\mathbb{Y})$ is clearly $\mathrm{Stab}_G(Y)$-equivariant, we deduce that every $Y$-loxodromic element of $\mathrm{Stab}_G(Y)$ must be $\mathcal{C}_K(\mathbb{Y})$-loxodromic.

\medskip \noindent
Thus, we have proved that $G$ acts acylindrically on the hyperbolic space $\mathcal{C}_K(\mathbb{Y})$ with unbounded orbits. It then follows from \cite[Theorem~1.1]{OsinAcyl} that $G$ is either acylindrically hyperbolic or virtually cyclic. 
\end{proof}

\subsection{Elementary subgroups and consequences of the WPD property}

\noindent Let $G$ be a group, and let $Y$ be a hyperbolic $G$-space. Following Bestvina and Fujiwara \cite{BF-WPD}, we say that an element $g\in G$ is \emph{WPD} in $Y$ if $g$ acts loxodromically on $Y$ and for every $R>0$ and every $y\in Y$, there exists $N\in\mathbb{N}$ such that the set $$\left\{h\in G \mid d(y,hy)\le R \text{~and~} d \left( g^Ny,hg^Ny \right) \le R\right\}$$ is finite.  

\medskip\noindent
Given a group $G$ and an element $g\in G$, we let $$E(g):= \left\{ h \in G \mid \exists n,m \in \mathbb{Z} \backslash \{0\}, hg^nh^{-1}= g^{m} \right\}.$$ The following result of Dahmani, Guirardel and Osin will be crucial in the present paper.

\begin{prop}[{Dahmani--Guirardel--Osin \cite[Lemma~6.5, Corollary~6.6, Theorem~6.8 and Proposition~2.8]{DGO}}]\label{prop:dgo}
Let $G$ be a group acting by isometries on a hyperbolic space $Y$, and let $g\in G$ be an element which is WPD in $Y$. Then $E(g)$ is virtually cyclic, and it is the unique maximal virtually cyclic subgroup of $G$ that contains $g$. 

\noindent In addition $E(g)$ is almost malnormal in $G$, i.e.\ for every $h\notin E(g)$, the intersection $E(g)\cap hE(g)h^{-1}$ is finite.
\end{prop}

\begin{remark}\label{rk:dgo}
Let $g\in G$ be an element which is WPD in $Y$. Notice that by definition, for every $k\in\mathbb{Z}\setminus\{0\}$, one has $E(g^k)=E(g)$. It follows that $E(g)$ is also the unique maximal virtually cyclic subgroup of $G$ that contains $g^k$. So $E(g)$ is the unique maximal virtually cyclic subgroup of $G$ that intersects $\langle g\rangle$ nontrivially. 
\end{remark}

\noindent
Given $g\in G$, we let $$\Fix_{\Aut(G)}(g):=\{\varphi\in\Aut(G) \mid\varphi(g)=g\}$$ and $$\Stab_{\Aut(G)}(E(g)):=\{\varphi\in\Aut(G) \mid \varphi(E(g))=E(g)\}.$$ 
We thank the referee for pointing out an elementary proof of the following fact.

\begin{lemma}\label{lemma:elementary-fixator}
Let $G$ be a group acting by isometries on a hyperbolic space $Y$, and let $g\in G$ be an element which is WPD in $Y$. Then $\Fix_{\Aut(G)}(g)$ is a finite index subgroup in $\Stab_{\Aut(G)}(E(g))$.
\end{lemma}

\begin{proof}
The definition of $E(g)$ ensures that $\Fix_{\Aut(G)}(g)$ is a subgroup of $\Stab_{\Aut(G)}(E(g))$. Since $E(g)$ contains only finitely many subgroups of index $[E(g):\langle g \rangle]$, the subgroup $\Stab_{\Aut(G)}(\langle g\rangle)$, made of automorphisms of $G$ that preserve $\langle g\rangle$, has finite index in $\Stab_{\Aut(G)}(E(g))$. As $\Fix_{\Aut(G)}(g)$ has index at most $2$ in $\Stab_{\Aut(G)}(\langle g\rangle)$, the lemma follows. 
\end{proof}

\noindent
In the case of group actions on simplicial trees, we will adopt the following equivalent definition of the WPD property (we refer to \cite[Corollary 4.3]{MR3368093} for a proof of this equivalence). Let $G$ be a group, let $T$ be a simplicial $G$-tree, and let $L,N \geq 0$. An element $g \in G$ is \emph{$(L,N)$-WPD} in $T$ if $g$ is loxodromic in $T$, and for every arc $I\subseteq\Axis_T(g)$ of length at least $L$, one has $|\Stab_G(I)|\le N$. The element $g$ is WPD in $T$ if there exist $L,N\ge 0$ such that $g$ is $(L,N)$-WPD in $T$. Notice that, if $G$ acts acylindrically on $T$, then its WPD elements are uniformly WPD (i.e.\ with uniform constants $L,N$).

\medskip \noindent 
The next lemma gives a control on the overlap between the axis of a WPD element in a tree $T$ and the characteristic set of every other element.

\begin{lemma}\label{lemma:overlap}
Let $G$ be a group acting by isometries on a simplicial tree $T$, equipped with the simplicial metric where every edge is assigned length $1$. Let $L\ge 1$ and $N\ge 0$, and let $g \in G$ be an element which is $(L,N)$-WPD in $T$. Let $h\in G$ be any element. If $\Axis_T(g)\cap\Char_T(h)$ has length at least $(N+2)L \max\{||g||_T,||h||_T\}$, then $h\in E(g)$.
\end{lemma}

\begin{proof}
First assume that $h$ acts elliptically on $T$, and fixes a segment that contains $(N+1)L$ fundamental domains of the axis of $g$. Then there exists a nondegenerate segment $I\subseteq T$ which contains $L$ fundamental domains of the action of $\langle g \rangle$ on the axis of $g$, and which is fixed by all elements $g^ihg^{-i}$ with $i\in\{0,\dots,N\}$. As every edge of $T$ has length $1$, the segment $I$ has length at least $L$. As $g$ is $(L,N)$-WPD in $T$, this implies that there exist $i\neq j$ such that $g^ihg^{-i}=g^jhg^{-j}$. In other words $g^{j-i}$ commutes with $h$, so $h\in E(g)$.

\medskip\noindent
Now assume that $h$ acts loxodromically on $T$. If $\Axis_T(g)\cap\Axis_T(h)$ contains $(N+2)L$ fundamental domains of both the axes of $g$ and $h$, then the commutators $[g^i,h]$, with $0 \leq i \leq N$, all fix a segment that contains at least $L$ fundamental domains of the axis of $g$. As above, there must exist $i \neq j$ such that
$$g^ihg^{-i} h^{-1}= [g^i,h]= [g^j,h]= g^jhg^{-j}h^{-1},$$
hence $g^{i-j}hg^{-(i-j)} = h$, and finally $h\in E(g)$. 
\end{proof}

\section{A general criterion for acylindrical hyperbolicity of $\Aut(G)$}\label{sec:criterion}

The goal of the present section is to establish a general criterion (Proposition~\ref{prop:SimpleOne} below) ensuring that the automorphism group of a given group $G$ is acylindrically hyperbolic. This criterion requires having a $G$-action on a simplicial tree, and a control on patterns of intersections of axes of certain sufficiently generic elements of $G$. 

\subsection{The persistence of long intersections property}

The following definition will be crucial throughout the paper.

\begin{de}\label{de:Persistence}
Let $G$ be a group, let $T$ be a $G$-tree, and let $g\in G$ be a $T$-loxodromic element. We say that $g$ has the \emph{persistence of long intersections property in $T$} if for every $C \geq 1$, there exists $n(C) \geq 1$ such that for every automorphism $\varphi \in \Aut(G)$ and every subset $\mathcal{X}\subseteq G$, if all elements in $\mathcal{X}$ are $T$-loxodromic and $$\Axis_T(g)\cap\bigcap_{h\in\mathcal{X}}\Axis_T(h)$$ contains a segment of length at least $n(C) \|g\|_T$, then all elements in $\varphi(\mathcal{X})$ are $T$-loxodromic and $$\Axis_T(\varphi(g))\cap\bigcap_{h\in\mathcal{X}}\Axis_T(\varphi(h))$$ contains a segment of length at least $C\| \varphi(g) \|_T$.
\end{de}

\subsection{A general criterion}

Given a subgroup $\mathscr{A}\subseteq\Aut(G)$, we let $\Stab_{\mathscr{A}}(E(g)):=\{\varphi\in\mathscr{A}\mid\varphi(E(g))= E(g)\}$. Given an element $g\in G$, we denote by $\ad_g\in\mathrm{Inn}(G)$ the conjugation by the element $g$.

\begin{prop}\label{prop:SimpleOne}
Let $G$ be a group and $\mathscr{A} \subseteq \mathrm{Aut}(G)$ a subgroup such that $\mathscr{A}$ is not virtually cyclic. Assume that there exist an element $g\in G$ with $\ad_g\in\mathscr{A}$, and a nonelementary simplicial minimal $G$-tree $T$ with the following properties:
\begin{enumerate}
	\item \textbf{Stable WPD:} There exist $L,N \geq 0$ such that, for every automorphism $\varphi \in \mathscr{A}$, the element $\varphi(g)$ is $(L,N)$-WPD in $T$.
	\item \textbf{Elementary fixator:} $\langle \mathrm{ad}_g \rangle$ has finite index in $\{ \varphi \in \mathscr{A} \mid \varphi(g)=g\}$.
	\item \textbf{Nielsen realisation:} $\langle \Stab_\mathscr{A}(E(g)) , \mathrm{Inn}(G)\rangle$ admits an action on $T$ whose restriction to $\mathrm{Inn}(G)$ coincides with the action induced by $\mathrm{Inn}(G) \simeq G/Z(G) \curvearrowright T$ (well-defined according to Lemma \ref{lem:center}).
	\item \textbf{Persistence of long intersections:}
All elements $g'\in\mathscr{A}\cdot g$ have the persistence of long intersections property in $T$.
\end{enumerate}
Then $\mathscr{A}$ is acylindrically hyperbolic. Moreover, $\mathrm{ad}_g$ is a generalised loxodromic element of $\mathscr{A}$. 
\end{prop}

\begin{remark}
Let us comment on the terminology \emph{Nielsen realisation} in this definition. As a consequence of the \emph{stable WPD} and \emph{elementary fixator} properties and Lemma~\ref{lemma:elementary-fixator}, the image of $\langle \Stab_\mathscr{A}(E(g)) , \mathrm{Inn}(G) \rangle$ in $\Out(G)$ is finite. Therefore, we are requiring that $T$ be an invariant tree for a certain finite subgroup $H$ of $\Out(G)$. When $G=\mathbb{F}_n$, the existence of a free simplicial $\mathbb{F}_n$-tree which is $H$-invariant (proved in \cite{Cul}) is the natural $\Out(\mathbb{F}_n)$-analogue of the classical Nielsen realization problem asking whether every finite subgroup of the mapping class group of a finite-type hyperbolic surface is realizable as the group of isometries of a certain hyperbolic surface (positively answered by Kerckhoff in \cite{Ker}).
\end{remark}

\begin{remark}
We warn the reader that in Assumption~4, the constant $n(C)$ that appears in the persistence of long intersections property is allowed to depend on $g'$. 
\end{remark}

\begin{proof}
We denote by $\rho_0:G\to\Isom(T)$ the given $G$-action on $T$. Let $\mathscr{D}$ be the set of all pairs $(\rho,x)$, where $\rho : G \to \mathrm{Isom}(T)$ is a $G$-action by isometries on $T$ and $x \in T$ is a basepoint, up to the following equivalence relation: $(\rho_1,x_1) \sim (\rho_2,x_2)$ if there exists an isometry $\iota:T\to T$ sending $x_1$ to $x_2$ and which is \emph{$(\rho_1,\rho_2)$-equivariant}, meaning that for every $x\in T$ and every $h\in G$, one has $\rho_2(h)(\iota(x))=\iota(\rho_1(h)(x))$. Given a pair $(\rho,x)$ as above, we denote by $[\rho,x]$ the equivalence class of $(\rho,x)$. The group $\mathscr{A}$ acts on $\mathscr{D}$ via $\varphi \cdot [\rho,x] = [\rho \circ \varphi^{-1},x]$. 

\medskip\noindent
We observe that given a homomorphism $\rho:G\to\Isom(T)$ and two distinct points $x_1,x_2\in T$, the pairs $(\rho,x_1)$ and $(\rho,x_2)$ are not equivalent: indeed, otherwise, there would exist an isometry $\iota:T\to T$ sending $x_1$ to $x_2$ and commuting with the $\rho$-action, contradicting Lemma~\ref{lem:center}.

\medskip \noindent
Given a $G$-action $\rho:G\to\mathrm{Isom}(T)$ and an element $h\in G$ which is loxodromic in $T$ for the $\rho$-action, we denote by $\Axis_\rho(h)$ the axis of $h$ in $T$ for the $\rho$-action. For simplicity of notation, we will simply write $\Axis(h)$ for $\Axis_{\rho_0}(h)$. 
For every $\varphi \in \mathscr{A}$, set $$Y_\varphi := \left\{ \left[\rho_0 \circ \varphi^{-1}, x\right] \mid x \in \Axis(g) \right\}\subseteq\cald,$$
where $g$ is the element coming from our assumptions. In view of the observation made in the above paragraph, every element of $Y_\varphi$ is written as $\left[\rho_0 \circ \varphi^{-1}, x\right]$ for a unique $x\in\Axis(g)$. We can (and shall) therefore endow $Y_\varphi$ with the metric 
$$\lambda_\varphi : ( [\rho_0 \circ \varphi^{-1}, x_1], [\rho_0 \circ \varphi^{-1} ,x_2]) \mapsto d_T(x_1,x_2).$$  We claim that $\lambda_\varphi$ only depends on $Y_\varphi$, and not on the choice of $\varphi$ itself. Indeed, if $Y_\varphi=Y_\psi$ and if we have $[\rho_0\circ\varphi^{-1},x_1]=[\rho_0\circ\psi^{-1},y_1]$ and $[\rho_0\circ\varphi^{-1},x_2]=[\rho_0\circ\psi^{-1},y_2]$, then there exists a $(\rho_0\circ\varphi^{-1},\rho_0\circ\psi^{-1})$-equivariant isometry $\iota:T\to T$ sending $x_1$ to $y_1$. Then $[\rho_0\circ\varphi^{-1},x_2]=[\rho_0\circ\psi^{-1},\iota(x_2)]$, showing that $y_2=\iota(x_2)$. As $\iota$ is an isometry of $T$, we have $d_T(x_1,x_2)=d_T(\iota(x_1),\iota(x_2))=d_T(y_1,y_2)$, showing that $\lambda_\varphi=\lambda_\psi$. 

\medskip\noindent
In particular, all spaces $Y_\varphi$ are isometric to the real line, whence $0$-hyperbolic. Our goal is to apply the criterion coming from the Bestvina--Bromberg--Fujiwara construction (Theorem~\ref{thm:BBF}) to the collection of all metric spaces $Y_{\varphi}$. 
Notice that for all $\varphi,\xi\in\mathscr{A}$, we have $\xi\cdot Y_\varphi=Y_{\xi\varphi}$. In addition, for all $x_1,x_2\in \Axis(g)$, we have 
$$\begin{array}{lcl} \lambda_{\xi \varphi} \left( \xi \cdot [\rho_0 \circ \varphi^{-1},x_1], \xi \cdot [\rho_0 \circ \varphi^{-1},x_2] \right) & = & \lambda_{\xi \varphi} \left( [\rho_0 \circ \varphi^{-1} \circ \xi^{-1},x_1],  [\rho_0 \circ \varphi^{-1} \circ \xi^{-1},x_2] \right) \\ \\ & = & d_T(x_1,x_2) \\ \\ & = & \lambda_{\varphi} \left( [\rho_0 \circ \varphi^{-1},x_1], [\rho_0 \circ \varphi^{-1},x_2] \right).  \end{array}$$
This precisely means that $\mathscr{A}$ acts on $\mathbb{Y}:=\{(Y_\varphi, \lambda_\varphi) \mid \varphi \in \mathscr{A}\}$ in a metric-preserving way (in the sense recalled above the statement of Theorem~\ref{thm:BBF}): indeed, this is shown by defining $\iota_\xi^{Y_\varphi}:Y_\varphi\to\xi\cdot Y_\varphi$ via $$\iota_\xi^{Y_\varphi}([\rho_0\circ\varphi^{-1},x])=[\rho_0\circ\varphi^{-1}\circ\xi^{-1},x]$$ (that this is well-defined follows from the observation made in the second paragraph of this proof). 

\medskip\noindent Recall from Assumption~1 (Stable WPD) that all elements of $G$ in the $\mathscr{A}$-orbit of $g$ are loxodromic in $T$ for the $\rho_0$-action. The following claim describes how the set $Y_{\varphi}$ depends on $\varphi$.

\begin{claim}\label{claim:equality-Y}
For all $\varphi_1,\varphi_2\in\mathscr{A}$, the following statements are equivalent.
\begin{enumerate}
\item $Y_{\varphi_1}=Y_{\varphi_2}$,
\item $\varphi_1(g)$ and $\varphi_2(g)$ have the same axis in $T$ for the $\rho_0$-action,
\item $g$ and $\varphi_1^{-1}\varphi_2(g)$ have the same axis in $T$ for the $\rho_0$-action,
\item $\varphi_1^{-1}\varphi_2(g)\in E(g)$,
\item $\varphi_1^{-1}\varphi_2\in\Stab_{\mathscr{A}}(E(g))$. 
\end{enumerate}
\end{claim}

\begin{proof}[Proof of Claim~\ref{claim:equality-Y}]
The equivalence $2\Leftrightarrow 3$ follows from Assumption~4 (Persistence of long intersections). We will now prove the equivalence of Assertions~1,~3,~4 and~5, and for that we can (and shall) assume without loss of generality that $\varphi_1=\id$; for simplicity of notation, we also replace $\varphi_2$ with $\varphi$. 

\medskip\noindent
We first prove that $1\Rightarrow 3$. Assume that $Y_\mathrm{Id}= Y_{\varphi}$. Then there exists a bijective map $p : \Axis(g) \to \Axis(g)$ such that, for every $x \in \Axis(g)$, we have $(\rho_0 \circ \varphi^{-1},x) \sim (\rho_0, p(x))$. In other words, there exists a $(\rho_0\circ\varphi^{-1},\rho_0)$-equivariant isometry $i_x: T \to T$ sending $x$ to $p(x)$. Notice that, for every $g'\in\mathscr{A}\cdot g$ and every $x \in \Axis(g')$, the isometry $i_x$ sends a $\langle \rho_0 \circ \varphi^{-1}(g') \rangle$-invariant geodesic to a $\langle \rho_0(g') \rangle$-invariant geodesic, hence $i_x \left( \Axis_{\rho_0 \circ \varphi^{-1}}(g') \right)= \Axis(g')$. As a consequence, for every $x\in\Axis(g)$, the element $p(x)=i_x(x)$ belongs to $i_x \left( \Axis(g) \right) = i_x \left( \Axis_{\rho_0 \circ \varphi^{-1}}(\varphi(g)) \right) = \Axis(\varphi(g)).$
Hence $$\Axis(g) = \{ p(x) \mid x \in \Axis(g) \} \subseteq \Axis(\varphi(g)),$$
and finally $\Axis(g)= \Axis(\varphi(g))$ (if one axis is contained in the other, then they are equal). In other words, $g$ and $\varphi(g)$ have the same axis in $T$ for the $\rho_0$-action. 

\medskip\noindent
We now prove that $3\Leftrightarrow 4$. The implication $3\Rightarrow 4$ follows from the fact that $g$ is WPD in $T$ (Assumption~1) together with Lemma~\ref{lemma:overlap} (applied with $h=\varphi(g)$). For the converse implication, as $g$ and $\varphi(g)$ have infinite order and both belong to the virtually cyclic group $E(g)$, there exist non-zero integers $n,m$ such that $g^n=\varphi(g)^m$, which implies that $g$ and $\varphi(g)$ have the same axis. 

\medskip\noindent
We now prove that $4\Leftrightarrow 5$. The implication $5 \Rightarrow 4$ is clear. To prove the converse $4\Rightarrow 5$, if $\varphi(g)\in E(g)$, then $\varphi(g)$ has a power contained in $\langle g\rangle$. In other words $\varphi(E(g))$ is a virtually cyclic subgroup that intersects $\langle g\rangle$ nontrivially. As $E(g)$ is the maximal virtually cyclic subgroup of $G$ that intersects $\langle g\rangle$ nontrivially (see Remark~\ref{rk:dgo}), it follows that $\varphi(E(g))= E(g)$, as desired. 

\medskip\noindent
There remains to prove that $5\Rightarrow 1$. By Assumption~3 (Nielsen realisation), the automorphism $\varphi$ defines an isometry $I_{\varphi}$ of $T$ which extends the $G$-action. For every $h\in G$ and every $y\in T$, one then has $I_{\varphi}(hy)=\varphi(h)I_\varphi(y)$. It follows that $(\rho_0,x) \sim (\rho_0 \circ \varphi^{-1}, I_\varphi^{-1}(x))$ for every $x \in \Axis(g)$, because $z \mapsto I_\varphi^{-1}(z)$ defines a $(\rho_0,\rho_0\circ\varphi^{-1})$-equivariant isometry $T \to T$. Since $I_\varphi^{-1}$ sends $\Axis(g)$ to $\Axis(\varphi^{-1}(g)) $, and since $\Axis(\varphi^{-1}(g)) = \Axis(g)$ (as a consequence of $4 \Rightarrow 3$), we deduce that $Y_{\id}=Y_{\varphi}$ as desired.
\end{proof}

\noindent For every $g'\in\mathscr{A}\cdot g$, we let $\mathrm{proj}_{\Axis(g)}\Axis(g')$ be the closest-point projection of $\Axis(g')$ onto $\Axis(g)$: this is a subinterval of $\Axis(g)$.  
Given $\varphi,\psi\in\mathscr{A}$, set
$$\pi_{Y_\varphi}(Y_\psi):= \left\{ [\rho_0 \circ \varphi^{-1},x] \mid x \in \mathrm{proj}_{\Axis(g)} \Axis(\varphi^{-1}\psi(g)) \right\}.$$
We observe that this is well-defined, i.e.\ it does not depend on the choices of $\varphi$ and $\psi$. Indeed, if $\psi_1,\psi_2$ are such that $Y_{\psi_1}=Y_{\psi_2}$, then it follows from Claim \ref{claim:equality-Y} that $\varphi^{-1}(\psi_1(g))$ and $\varphi^{-1}(\psi_2(g))$ both belong to $\varphi^{-1}\psi_1(E(g))$, and therefore they have the same axis in $T$ for the $\rho_0$-action as $\varphi^{-1} \psi_1 (E(g))$ is virtually cyclic. This shows that our definition does not depend on the choice of $\psi$. Furthermore, if $\varphi_1,\varphi_2$ are such that $Y_{\varphi_1}=Y_{\varphi_2}$, then 
$$\left\{ [\rho_0 \circ \varphi_1^{-1},x] \mid x \in \mathrm{proj}_{\Axis(g)} \Axis(\varphi_1^{-1}\psi(g)) \right\}$$ is equal to $$\left\{[\rho_0\circ\varphi_2^{-1},\varphi_2^{-1}\circ\varphi_1\cdot x] \mid x\in\mathrm{proj}_{\Axis(g)}\Axis(\varphi_1^{-1}\psi(g))\right\},$$ where we think of $\varphi_2^{-1} \circ \varphi_1$ as an isometry of $T$ preserving the axis of $g$ (as allowed by Assumption 3 (Nielsen realisation)). Consequently, our subset is also equal to $$\left\{[\rho_0\circ\varphi_2^{-1},y] \mid y\in\mathrm{proj}_{\Axis(g)}\Axis(\varphi_2^{-1}\psi(g))\right\},$$
showing that our definition does not depend on the choice of $\varphi$.

\medskip\noindent Notice that
$$\begin{array}{lcl} \xi \cdot \pi_{Y_{\varphi}}(Y_\psi) & = & \left\{ [\rho_0 \circ \varphi^{-1} \circ \xi^{-1},x] \mid x \in \mathrm{proj}_{\Axis(g)} \Axis(\varphi^{-1}\psi(g)) \right\} \\ \\ & = &  \left\{ [\rho_0 \circ \varphi^{-1} \circ \xi^{-1},x] \mid x \in \mathrm{proj}_{\Axis(g)} \Axis( (\xi \varphi)^{-1} (\xi \psi)(g)) \right\} \\ \\ & = & \pi_{Y_{\xi \varphi}}( Y_{\xi \psi}). \end{array}$$
 Thus, the third assumption from Theorem~\ref{thm:BBF} (Equivariance) holds.

\medskip \noindent
For all $X,Y,Z \in \mathbb{Y}$, define
$$d_Y(X,Z):= \mathrm{diam}_Y \left( \pi_Y(X) \cup \pi_Y(Z) \right).$$
Notice that for all $\varphi,\psi,\xi\in\mathscr{A}$, one has
$$\begin{array}{lcl} d_{Y_\varphi}(Y_\psi, Y_\xi) & = & \mathrm{diam}_{\lambda_\varphi} \left( \pi_{Y_\varphi}(Y_\psi) \cup \pi_{Y_{\varphi}}(Y_\xi) \right) \\ \\ & = & \mathrm{diam}_{\lambda_\varphi} \left\{ [\rho_0 \circ \varphi^{-1},x] \mid x \in \mathrm{proj}_{\Axis(g)} \left( \Axis(\varphi^{-1}\psi(g)) \cup  \Axis(\varphi^{-1}\xi(g)) \right) \right\} \\ \\ & = & \mathrm{diam}_T \left(  \mathrm{proj}_{\Axis(g)} \Axis(\varphi^{-1}\psi(g)) \cup \mathrm{proj}_{\Axis(g)} \Axis(\varphi^{-1}\xi(g)) \right). \end{array}$$

\noindent  We now verify that the projection axioms from Theorem~\ref{thm:BBF} hold. 
Let $L\ge 1$ and $N\ge 0$ be such that for every $\varphi\in\mathscr{A}$, the element $\varphi(g)$ is $(L,N)$-WPD in $T$. Let $K:=(N+2)L$ (as in Lemma~\ref{lemma:overlap}). For every $C \geq 0$ and every $g' \in \mathscr{A} \cdot g$, let $n(C,g')$ denote the constant coming from Assumption~4 (Persistence of long intersections) as in Definition \ref{de:Persistence}. In the sequel, when we write $\|g\|_T$, we refer to the translation length of $g$ for the $\rho_0$-action on $T$. 

\medskip\noindent 
\textbf{Condition (P0).} This follows from the following more general claim when applied to $\varphi=\mathrm{Id}$.

\begin{claim}\label{claim:condition-P0}
For every $\varphi\in\mathscr{A}$, there exists $D(\varphi) \geq 0$ such that, for every $\psi \in \mathscr{A}$ satisfying $Y_\psi \neq Y_\varphi$, one has $\mathrm{diam}_T \left( \mathrm{proj}_{\Axis(\varphi(g))}(\Axis(\psi(g))) \right) \leq D(\varphi)$.
\end{claim}

\begin{proof}[Proof of Claim~\ref{claim:condition-P0}]
Let $D(\varphi):=n(K\|g\|_T,\varphi(g))\|\varphi(g)\|_T$. Assume that the intersection between the axes of $\psi(g)$ and $\varphi(g)$ in $T$ (for the $\rho_0$-action) has length greater than $D(\varphi)$. It follows from Assumption~4 (Persistence of long intersections) that the intersection between the axes of $g$ and $\psi^{-1}\varphi(g)$ has length at least $K \|g \|_T \| \psi^{-1}\varphi(g)\|_T$ (whence at least $K\max\{\|g\|_T,\|\psi^{-1}\varphi(g)\|_T\}$ since all translation lengths of loxodromic elements are at least $1$). Lemma~\ref{lemma:overlap} therefore implies that $\psi^{-1}\varphi(g)\in E(g)$. Claim~\ref{claim:equality-Y} concludes that $Y_{\psi}=Y_{\varphi}$.
\end{proof}

\noindent
From now on we let $D:=D(\mathrm{Id})$: this bounds the diameter of the projection of the axis of $\varphi(g)$ onto the axis of $g$ for every $\varphi\in\mathscr{A}$ such that these axes are distinct. 
Let $$\theta := \left( 2D+ \max(1, n(2D+1,g), n(K,g)) \right) \|g\|_T.$$ We have just proved that Condition~(P0) from Theorem~\ref{thm:BBF} holds for this choice of $\theta$, and we will now check Conditions~(P1) and~(P2).

\medskip \noindent
\textbf{Condition (P1).} Let $\varphi,\psi,\xi\in\mathscr{A}$ be such that $Y_\varphi, Y_\psi,Y_\xi$ are pairwise distinct. Assume that $d_{Y_{\varphi}}(Y_\psi,Y_\xi)>\theta$, in particular $d_{Y_{\varphi}}(Y_\psi,Y_\xi) > \left(n(2D+1,g)+2D \right) \|g\|_T$.

\medskip \noindent
Then the union of the projections of $\Axis(\varphi^{-1} \psi(g))$ and $\Axis(\varphi^{-1} \xi(g))$ onto $\Axis(g)$ has diameter greater than $\left(n(2D+1,g)+2D \right) \|g\|_T$. As these projections have diameter at most $D$, the axes of $\varphi^{-1}\psi(g)$ and $\varphi^{-1}\xi(g)$ are disjoint, and their bridge contains a subsegment of $\Axis(g)$ of length greater than $\left( n(2D+1,g)+2D \right) \|g\|_T-2D \geq n(2D+1,g) \|g\|_T$. Lemma~\ref{lemma:culler-morgan} therefore implies that $\varphi^{-1}(\psi(g)\xi(g))$ is loxodromic (for the $\rho_0$-action) and its axis intersects $\Axis(g)$ along a segment of length at least $n(2D+1,g) \|g\|_T$.

\medskip \noindent
It follows from Assumption~4 (Persistence of long intersections) applied to $g'=g$ (by choosing $h=\varphi^{-1}(\psi(g)\xi(g))$ and the automorphism $\psi^{-1}\varphi$ in Definition~\ref{de:Persistence}) 
that $g\psi^{-1}\xi(g)$ is loxodromic and that the intersection of the axes of $g\psi^{-1}\xi(g)$ and $\psi^{-1}\varphi(g)$ (for the $\rho_0$-action) has diameter at least $2D+1$ (in particular these axes intersect each other). Applying Lemma~\ref{lemma:action-on-tree} with $h=\psi^{-1}\xi(g)$ and $h'=\psi^{-1}\varphi(g)$ shows that the distance between the projections of $\Axis(\psi^{-1}\xi(g))$ and of $\Axis(\psi^{-1}\varphi(g))$ onto $\Axis(g)$ is at most $\|g\|_T$. As these projections have diameter at most $D$, we conclude that $d_{Y_\psi}(Y_\varphi,Y_{\xi}) \leq 2D+\|g\|_T\leq\theta$.
\qed

\medskip \noindent
\textbf{Condition (P2).} Let $\varphi,\psi\in\mathscr{A}$ be such that $Y_{\varphi}\neq Y_{\psi}$. Assume for contradiction that $$\{ Y_\xi \neq Y_{\varphi}, Y_{\psi} \mid d_{Y_\xi}(Y_{\varphi}, Y_\psi)> \theta\}$$ is infinite. In particular, $$\{ Y_\xi \neq Y_{\varphi}, Y_{\psi} \mid d_{Y_\xi}(Y_{\varphi}, Y_\psi)> (n(K,g)+2D) \| g \|_T\}$$ is infinite.

\medskip \noindent
Let $(\xi_i)_{i\in\mathbb{N}}\in\mathscr{A}^{\mathbb{N}}$ be an infinite sequence such that the spaces $Y_{\xi_i}$ are pairwise distinct and all belong to the above set. Fix some $i \geq 1$. Then the union of the projections onto $\Axis(g)$ of $\Axis(\xi_i^{-1} \varphi(g))$ and $\Axis(\xi_i^{-1}\psi(g))$ has diameter at least $(n(K,g)+2D) \|g\|_T$. As these projections have diameter at most $D$, it follows using Lemma~\ref{lemma:culler-morgan} that $$\Axis(g)\cap\bigcap_{n,m\in\mathbb{Z}\setminus\{0\}}\Axis(\xi_i^{-1}(\varphi(g)^n \psi(g)^m))$$ has diameter at least $$(n(K,g)+2D) \|g\|_T-2D \geq n(K,g) \|g\|_T.$$  
It follows from Assumption~4 (Persistence of long intersections) that for every $n,m\in\mathbb{Z}\setminus\{0\}$, the element $\varphi(g)^n\psi(g)^m$ is loxodromic (for the $\rho_0$-action) and that $$\Axis(\xi_i(g))\cap\bigcap_{n,m\in\mathbb{Z}\setminus\{0\}}\Axis(\varphi(g)^n \psi(g)^m)$$ has diameter at least $K\|\xi_i(g) \|_T$. 
Using Lemma~\ref{lemma:tree-extra}, it follows that the axes of $\varphi(g)$ and $\psi(g)$ are disjoint, and that the axis of $\xi_i(g)$ intersects the bridge between $\Axis(\varphi(g))$ and $\Axis(\psi(g))$ along a subsegment $\sigma_i$ of length at least $K\|\xi_i(g)\|_T$. 

\medskip \noindent
Up to extracting a subsequence,
we may suppose without loss of generality that $\sigma_i$ does not depend on $i$. Let $\sigma$ denote this subsegment. Also, because the translation lengths of the elements $\xi_i(g)$ are bounded by the distance between $\Axis(\varphi(g))$ and $\Axis(\psi(g))$, up to extracting a subsequence, we may suppose without loss of generality that the elements $\xi_i(g)$ all have the same translation length. Now, for every $i \geq 2$, the axes of $\xi_1(g)$ and $\xi_i(g)$ have an overlap of length at least $K\|\xi_1(g)\|_T$. As $\xi_1(g)$ is $(L,N)$-WPD (Assumption~1), it follows from Lemma~\ref{lemma:overlap} that $\xi_i(g)\in \xi_1(E(g))$, i.e.\ $\xi_1^{-1}\xi_i(g)\in E(g)$. By Claim~\ref{claim:equality-Y}, it follows that $Y_{\xi_1}=Y_{\xi_i}$, a contradiction.
\qed 

\medskip \noindent
Finally, the stabiliser of $Y_\mathrm{Id}$ coincides with $\mathrm{Stab}_\mathscr{A}(E(g))$ according to Claim~\ref{claim:equality-Y}, so it is virtually cyclic in view of Assumption~2 (Elementary fixator) together with Lemma~\ref{lemma:elementary-fixator}, and it acts with unbounded orbits on $Y_\mathrm{Id}$. So the unboundedness and local acylindricity conditions of Theorem~\ref{thm:BBF} are clear. And the global acylindricity will follow from our next claim:

\begin{claim}\label{claim:acylindricity}
Let $\varphi,\psi\in\mathscr{A}$. If $Y_\psi \neq Y_\varphi$, then $\mathrm{Stab}_{\mathscr{A}}(Y_\psi) \cap \mathrm{Stab}_{\mathscr{A}}(Y_\varphi)$ has cardinality at most $[\Stab_{\mathscr{A}}(E(g)): \langle \mathrm{ad}_g \rangle]$.
\end{claim}

\begin{proof}[Proof of Claim~\ref{claim:acylindricity}]
The following proof was provided to us by the anonymous referee, simplifying our original argument. We assume without loss of generality that $\psi=\mathrm{Id}$. If $\mathrm{Stab}_{\mathscr{A}}(Y_\mathrm{Id}) \cap \mathrm{Stab}_{\mathscr{A}}(Y_\varphi)$ has cardinality bigger than $[\Stab_{\mathscr{A}}(E(g)): \langle \mathrm{ad}_g \rangle]$, then $\langle \ad_g\rangle$ has nontrivial intersection with $\Stab_{\mathscr{A}}(Y_\varphi)=\Stab_{\mathscr{A}}(\varphi(E(g)))$ (this last equality comes from Claim~\ref{claim:equality-Y}). Hence, there exists $k>0$ such that $g^k E(\varphi(g))g^{-k}=E(\varphi(g))$. It follows that $g^k\in E(\varphi(g))$, and Remark~\ref{rk:dgo} then implies that $\varphi^{-1}(g)\in E(g)$. Using Claim~\ref{claim:equality-Y} again shows that $Y_{\varphi}=Y_{\mathrm{Id}}$.
\end{proof}

\noindent
By Assumption~2 (Elementary fixator), the group $\langle\ad_g\rangle$ has finite index in $\{\varphi\in\mathscr{A}|\varphi(g)=g\}$, and by Lemma~\ref{lemma:elementary-fixator} this in turn has finite index in $\Stab_{\mathscr{A}}(E(g))$. In other words $[\Stab_{\mathscr{A}}(E(g)):\langle \mathrm{ad}_g\rangle]$ is finite. It therefore follows from Claim~\ref{claim:acylindricity} that Condition~5 from Theorem~\ref{thm:BBF} (Global acylindricity) holds (with $N=2$ and $B=[\Stab_{\mathscr{A}}(E(g)):\langle \mathrm{ad}_g\rangle]$). Conditions~1 (Projection axioms), 2 (Unboundedness), 3 (Isometric action) and 4 (Local acylindricity) have been checked above. Thus, we have proved that Theorem~\ref{thm:BBF} applies, and as $\mathscr{A}$ is not virtually cyclic by assumption, we conclude that $\mathscr{A}$ is acylindrically hyperbolic. Moreover, because $\mathrm{ad}_g$ is a $Y_{\mathrm{Id}}$-loxodromic element of the stabiliser of $Y_\mathrm{Id}$, it also follows from Theorem \ref{thm:BBF} that $\mathrm{ad}_g$ is a generalised loxodromic element of $\mathscr{A}$.
\end{proof}

\subsection{The case of an $\mathrm{Aut}(G)$-invariant tree}

\noindent
The simplest case where Proposition \ref{prop:SimpleOne} applies is when the $G$-tree $T$ is \emph{$\mathscr{A}$-invariant}, i.e.\ the action $\mathrm{Inn}(G) \simeq G/Z(G) \curvearrowright T$ extends to an isometric action $\mathscr{A} \curvearrowright T$. This may happen when $G$ admits a canonical JSJ decomposition. For instance, this strategy has been successfully applied in \cite{MR4011668} to one-ended hyperbolic groups, and it will be used in the next section in the broader context of one-ended relatively hyperbolic groups. A direct proof of the following criterion is essentially contained in the argument given in \cite{MR4011668}, but it can also be deduced from Proposition \ref{prop:SimpleOne}. 

\begin{prop}\label{prop:JSJcanonical}
Let $G$ be a group, let $\mathscr{A} \subseteq \mathrm{Aut}(G)$ be a subgroup which is not virtually cyclic, and let $T$ be an $\mathscr{A}$-invariant minimal simplicial $G$-tree. Assume that the $G$-action on $T$ is nonelementary, and that there exists an element $g\in G$ which is WPD in $T$ such that $\langle \mathrm{ad}_g \rangle$ is contained in $\mathscr{A}$ and has finite index in $\{ \varphi \in \mathrm{Aut}(G) \mid \varphi(g)=g\}$. Then $\mathscr{A}$ is acylindrically hyperbolic.
\end{prop}

\begin{proof}
Conditions 1 (Stable WPD), 3 (Nielsen realisation) and 4 (Persistence of long intersections) in Proposition \ref{prop:SimpleOne} are clearly sastisfied as $T$ is $\mathscr{A}$-invariant. Consequently, Proposition \ref{prop:SimpleOne} applies and the desired conclusion follows.
\end{proof}

\subsection{A word on the Nielsen realisation assumption}

\noindent 
Interestingly, for the price of a slightly weaker conclusion, the Nielsen realisation assumption can often be removed from the conditions of our criterion if one has more information about the outer automorphism group of $G$, namely if we know that $\mathrm{Out}(G)$ is virtually torsion-free or residually finite. Having a residually finite outer automorphism group is widespread, as a consequence of the strategy introduced by Grossman in \cite{MR405423}. For instance, Minasyan and Osin proved that the outer automorphism group of every finitely generated residually finite infinitely-ended group is residually finite \cite{MO}. See also \cite{MR2231162, MR2593663, MR3605030} for other instances of such statements. 

\medskip \noindent
The price to pay in the conclusion is that without the Nielsen realisation assumption, we can only prove that $\mathrm{Aut}(G)$ contains a finite-index subgroup which is acylindrically hyperbolic. The stability of acylindrical hyperbolicity under finite-index overgroups is still open to our knowledge (see the discussion in \cite{MR3968890}), so \emph{a priori} we cannot conclude that the entire automorphism group is acylindrically hyperbolic. Nevertheless, many of the interesting properties which can be deduced from being acylindrically hyperbolic, such as the existence of uncountably many normal subgroups \cite[Theorem 2.33]{DGO}, pass from a finite-index subgroup to the overgroup, so proving that a group is virtually acylindrically hyperbolic remains of interest.

\begin{prop}\label{prop:SimpleTwo}
Let $G$ be a group such that $\mathrm{Out}(G)$ is virtually torsion-free or residually finite, and $\mathscr{A} \subseteq \mathrm{Aut}(G)$ a subgroup which is not virtually cyclic. Assume that there exist an element $g\in G$ with $\ad_g\in\mathscr{A}$, and a nonelementary simplicial minimal $G$-tree $T$ with the following properties:
\begin{enumerate}
	\item \textbf{Stable WPD:} There exist $L,N \geq 0$ such that, for every automorphism $\varphi \in \mathscr{A}$, the element $\varphi(g)$ is $(L,N)$-WPD in $T$.
	\item \textbf{Elementary fixator:} $\langle \mathrm{ad}_g \rangle$ has finite index in $\{ \varphi \in \mathscr{A} \mid \varphi(g)=g\}$.
	\item \textbf{Persistence of long intersections:}
	All elements in $\mathscr{A}\cdot g$ have the persistence of long intersections property in $T$. 
\end{enumerate}
Then $\mathscr{A}$ is virtually acylindrically hyperbolic.
\end{prop}

\begin{proof}
By Lemma~\ref{lemma:elementary-fixator}, the group $\Stab_{\mathscr{A}}(E(g))$ contains $\Fix_{\mathscr{A}}(g)$ as a finite-index subgroup. It thus follows from Assumption~2 (Elementary fixator) that $\Stab_{\mathscr{A}}(E(g))$ has finite image in $\Out(G)$. As $\mathrm{Out}(G)$ is virtually torsion-free or residually finite, there exists a finite-index subgroup $\mathscr{A}_- \subseteq \mathscr{A}$ such that $\Stab_{\mathscr{A}_-}(E(g))$ is contained in $\mathrm{Inn}(G)$ (and with $\ad_g\in\mathscr{A}_-$). Now, we want to apply Proposition \ref{prop:SimpleOne} to $\mathscr{A}_-$ to deduce that $\mathscr{A}_-$ is acylindrically hyperbolic. Notice that the Nielsen realisation condition from Proposition~\ref{prop:SimpleOne} is now obviously satisfied, and the other conditions are hypotheses in our proposition. Consequently, Proposition~\ref{prop:SimpleOne} applies and the desired conclusion follows.
\end{proof}

\section{Relatively hyperbolic groups: the one-ended case}\label{sec:rh-one-ended}

In this section, we establish Theorem~\ref{thm:IntroRHacyl} from the introduction in the case where $G$ is one-ended relative to the collection $\mathcal{P}$ of peripheral subgroups. We refer the reader to \cite{Bow} for the definition of hyperbolicity of a group relative to a collection of subgroups; we say that a group is \emph{relatively hyperbolic} if there exists a finite collection $\mathcal{P}$ of proper subgroups such that $G$ is hyperbolic relative to $\mathcal{P}$.

\subsection{Nowhere elliptic elements}

Let $G$ be a group and $\mathcal{P}$ a collection of subgroups. A subgroup $H\subseteq G$ is \emph{$\calp$-elementary} if it is either virtually cyclic (possibly finite) or conjugate into a subgroup in $\calp$. A subgroup $H\subseteq G$ is an \emph{arc stabiliser} in a real $G$-tree $T$ if $H$ fixes a nondegenerate segment of $T$ pointwise.

\begin{de}\label{de:nowhere-elliptic}
Let $G$ be a group. An element $g\in G$ is \emph{nowhere elliptic in $\calz$-trees} if $g$ is not elliptic in any real $G$-tree with (finite or) virtually cyclic arc stabilisers. Given a collection $\calp$ of subgroups of $G$, we say that $g$ is \emph{nowhere elliptic in $\calz\calp$-trees} if $g$ is not elliptic in any real $G$-tree with $\calp$-elementary arc stabilisers.
\end{de}

\noindent
The main goal of this section is to construct elements that are nowhere elliptic in $\calz\calp$-trees, for specific families $\calp$ of subgroups, including peripheral subgroups in relatively hyperbolic groups. To be precise, the families of subgroups we are interested in are as follows -- we mention that in the sequel of the paper, we will actually only work with Property~(RAH2) (which also appears in \cite[Definition~4.21]{AutRAAG} under a different terminology) and never use Proposition~\ref{prop:RAH}, but we include it with a proof sketch mainly to justify our terminology.

\begin{prop}\label{prop:RAH}
Let $G$ be a group and $\mathcal{P}$ a finite collection of subgroups. The following assertions are equivalent:
\begin{description}
	\item[(RAH1)] $G$ admits an unbounded acylindrical action on a hyperbolic space such that each subgroup in $\mathcal{P}$ has bounded orbits;
	\item[(RAH1')] $G$ admits an unbounded acylindrical action on a quasi-tree such that each subgroup in $\mathcal{P}$ has bounded orbits;
	\item[(RAH2)] $G$ admits an action on a hyperbolic space with at least one WPD element such that each subgroup in $\mathcal{P}$ has bounded orbits;
	\item[(RAH2')] $G$ admits an action on a quasi-tree with at least one WPD element such that each subgroup in $\mathcal{P}$ has bounded orbits;
	\item[(RAH3)] $G$ admits an action on a geodesic metric space with at least one WPD contracting element such that each subgroup in $\mathcal{P}$ has bounded orbits.
\end{description}
When these equivalent conditions are satisfied, we say that $G$ is \emph{acylindrically hyperbolic relative to $\mathcal{P}$}.
\end{prop}

\noindent
Recall that an isometry $g \in \mathrm{Isom}(X)$ of a geodesic metric space $X$ is \emph{contracting} if there exists a point $x \in X$ such that $n \mapsto g^nx$ defines a quasi-isometric embedding $\mathbb{Z} \to X$ and such that the nearest-point projection onto the orbit $\langle g \rangle \cdot x$ of every metric ball disjoint from $\langle g \rangle \cdot x$ is has diameter bounded by a uniform constant (referred to as the \emph{contraction constant}). For $\mathcal{P}= \emptyset$, the equivalences $(RAH1) \Leftrightarrow (RAH2)$, $(RAH2) \Leftrightarrow (RAH3)$ and $(RAH1) \Leftrightarrow (RAH1')$ are proved respectively in \cite[Theorem~1.2]{OsinAcyl}, \cite[Theorem~H]{BBF} and \cite{MR3685605}. Proposition~\ref{prop:RAH} can be recovered from these arguments. 

\begin{proof}[Sketch of proof of Proposition \ref{prop:RAH}.]
The implications $(RAH1) \Rightarrow (RAH2)$ and $(RAH1') \Rightarrow (RAH2')$ follow from the observation that every loxodromic element is WPD when the action acylindrical and that loxodromic elements do exist as a consequence of \cite[Theorem~1.1]{OsinAcyl}; the implications $(RAH1') \Rightarrow (RAH1)$ and $(RAH2') \Rightarrow (RAH2)$ are clear; and the implication $(RAH2) \Rightarrow (RAH3)$ follows from the observation that loxodromic isometries in hyperbolic spaces are always contracting (see for instance \cite[Lemma~3.1]{MR3849623}). 

\medskip \noindent
Now, assume that $(RAH3)$ holds, i.e. $G$ acts on a geodesic metric space $X$ with a WPD contracting element $g \in G$ such that each subgroup in $\mathcal{P}$ is elliptic. We want to prove that $(RAH1')$ holds. Following \cite[Section 5.2]{BBFS}, we assume without loss of generality that $g$ acts as a translation on a bi-infinite geodesic line $\gamma$ (this is always possible up to replacing $X$ by a quasi-isometric space, as explained in the proof of \cite[Proposition~6(2)]{BF-WPD}). We will also assume without loss of generality that $\gamma$ is $EC(g)$-invariant where 
$$EC(g):= \{h \in G \mid \text{$\gamma$ and $h\gamma$ at finite Hausdorff distance} \}$$ (it follows from the fact that $g$ is a WPD element that $EC(g)=E(g)$). Let $\mathbb{Y}$ denote the collection of the $G$-translates of $\gamma$, and, for all $A,B \in \mathbb{Y}$, let $\pi_A(B)$ denote the nearest-point projection of $B$ onto $A$. From Theorem \ref{thm:BBF}, Assertion 1 (Projection axioms) is satisfied according to \cite[Theorem 5.2]{BBFS}; Assertions 2 (Unboundedness) and 3 (Equivariance) are satisfied by construction; and Assertion 4 (Local acylindricity) follows from the fact that $g$ is WPD. In order to verify Assertion~5 (Global acylindricity), obseve that, if $\gamma_1$ and $\gamma_2$ are two distinct $G$-translates of $\gamma$, then $\mathrm{Stab}(\gamma_1) \cap \mathrm{Stab}(\gamma_2)$ is finite. Indeed, the nearest-point projection of $\gamma_2$ onto $\gamma_1$ is bounded (because $g$ is contracting), so $\mathrm{Stab}(\gamma_1) \cap \mathrm{Stab}(\gamma_2)$ stabilises $\gamma_1$ and has bounded orbits. The desired conclusion follows from the fact that $\mathrm{Stab}(\gamma_1)$ acts properly on $\gamma_1$ since $\langle g \rangle$ has finite index $EC(g)$ (because $g$ is contracting and WPD). Thus, the assumptions of Theorem \ref{thm:BBF} are satisfied. By reproducing its proof, we define a new projection map $\pi'$ and new maps $\delta_Y$, and we construct a hyperbolic space $\mathcal{C}_K(\mathbb{Y})$ on which $G$ acts acylindrically with $g$ loxodromic. Actually, $\mathcal{C}_K(\mathbb{Y})$ is a quasi-tree according to \cite[Theorem~6.6]{BBFS}.

\medskip \noindent
In order to conclude the proof of our proposition, we claim that every subgroup in $\mathcal{P}$ has bounded orbits in $\mathcal{C}_K(\mathbb{Y})$. Fix a point $x \in \gamma$ and set 
$$D:= \max \{ \text{diameter of $H \cdot x$ in $X$} \mid H \in \mathcal{P}\}.$$ 
Also, let $C \geq 0$ denote the contraction constant of $\gamma$. If we choose $K>D+2(C+\theta)$, it follows from \cite[Theorem 6.3]{BBFS} combined with the projection axioms satisfied by the maps $\delta_Y$ that, given an arbitrary element $h$ in a subgroup in $\mathcal{P}$, we have
$$d_\mathcal{C}(x,hx) \leq 2 d_\gamma(x,hx) +3 \leq 2D+3$$
because $\delta_Y(x,hx) \leq d_Y(x,hx)+2\theta = d_X(x,hx)+ 2(C+\theta) \leq D+2(C+\theta)$ for every $Y \in \mathbb{Y}$. Therefore, every subgroup in $\mathcal{P}$ has a bounded orbit, concluding the proof of our proposition.
\end{proof}

\noindent
Our main construction of nowhere elliptic elements is the following.

\begin{lemma}\label{lemma:existence-nowhere-elliptic}
Let $G$ be a group which is not virtually cyclic, and $\calp$ a collection of subgroups. If $G$ is acylindrically hyperbolic relative to $\mathcal{P}$, then there exists an element $g\in G$ which is nowhere elliptic in $\calz\calp$-trees.
\end{lemma}

\begin{proof}
Fix a finite generating set $X_0$ of $G$, and set $X= (X_0\cup\{1\})^4$. Let $G$ act on a hyperbolic space $Y$ with at least one WPD element such that all the subgroups in $\mathcal{P}$ are elliptic. As $G$ is not virtually cyclic, the set
$$\left\{ E(g) \mid g \in G \ \text{~is WPD in~} Y \right\}$$
is an infinite collection of maximal virtually cyclic subgroups. Notice that if $g,g'\in G$ are two WPD elements for the $G$-action on $Y$ with $E(g)\neq E(g')$, then $E(g)\cap E(g')$ is finite. Therefore, every element $s\in X$ of infinite order is contained in at most one subgroup of $G$ of the form $E(g)$ with $g$ WPD in $Y$. It follows that we can (and shall) choose a WPD element $g \in G$ such that for every $s\in X$ of infinite order, one has $s \notin E(g)$. From \cite[Theorem 8.7]{DGO}, we know that, up to replacing $g$ with a sufficiently large power, the normal subgroup $\langle \langle g \rangle \rangle$ is free, and every nontrivial element of $\langle\langle g\rangle\rangle$ acts loxodromically on $Y$. Consequently, $\langle \langle g \rangle \rangle$ intersects every subgroup of $G$ which is conjugate to a subgroup in $\mathcal{P}$ trivially. As a subgroup of $\langle \langle g \rangle \rangle$, the group $H:= \langle sgs^{-1}, \ s \in X \rangle$  is free; moreover, its rank is at least two because otherwise, for every $s \in X$, we would have $sE(g)s^{-1} \subseteq E(g)$, which would imply $s \in E(g)$ as $E(g)$ is almost malnormal (Proposition~\ref{prop:dgo}). 

\medskip\noindent According to \cite{MR2585579} (see also \cite{Solie, MR3367520, GuirardelLevittRandomSplit}), there exists $h \in H$ which is loxodromic in every nontrivial simplicial $H$-action on a simplicial tree with cyclic (possibly finite) edge stabilisers. We will prove that $h$ is nowhere elliptic in $\calz\calp$-trees, which will conclude our proof.

\medskip \noindent
So let $T$ be an $\mathbb{R}$-tree equipped with a nontrivial isometric action of $G$ with $\calp$-elementary arc stabilisers, and let us prove that $h$ is loxodromic in $T$. Up to replacing $T$ by its minimal $G$-invariant subtree, we will assume that the $G$-action on $T$ is minimal. We will first prove that $H$ does not fix any point in $T$. If $g$ acts loxodromically on $T$, there is nothing to prove as $g \in H$. We thus suppose that $\mathrm{Fix}(g)$ is non-empty. 

\medskip\noindent We claim that $(X_0\cup\{1\})^2$ contains an element which acts loxodromically on $T$. This is obvious if $X_0$ contains a loxodromic isometry, so assume that all the elements of $X_0$ are elliptic. As $G=\langle X_0\rangle$ does not fix any point in $T$,  there exist two distinct elements $x,x'\in X_0$ such that $\mathrm{Fix}(x) \cap \mathrm{Fix}(x') = \emptyset$. Then the product $xx'$ is loxodromic by Lemma \ref{lemma:culler-morgan}, which proves our claim. 

\medskip\noindent Let now $s\in (X_0\cup\{1\})^2$ be a loxodromic element, as provided by the above paragraph. If $\mathrm{Fix}(g)$ is reduced to a single point, then $\mathrm{Fix}(g)$ and $\mathrm{Fix}(sgs^{-1})=s \mathrm{Fix}(g)$ are disjoint, so that $g \cdot sgs^{-1}$ defines a $T$-loxodromic element in $H$. From now on, we assume that $\mathrm{Fix}(g)$ is not reduced to a single point. If the axis of $s$ has a subsegment of length bigger than $\|s\|_T$ contained in $\mathrm{Fix}(g)$, then the commutator $[g,s]$ fixes a non-degenerate arc $I \subseteq \mathrm{Fix}(g)$. Notice that $\Stab(I)$ is a $\calp$-elementary subgroup that contains $g$; as $\langle \langle g \rangle \rangle$ intersects the groups in $\mathcal{P}$ trivially, it follows that $\Stab(I)$ must be virtually cyclic. As a consequence, $\mathrm{Stab}(I) \subseteq E(g)$. Hence the element $sg^{-1}s^{-1} = g^{-1}[g,s]$ belongs to $g^{-1} E(g)=E(g)$,
so $\langle g \rangle \subseteq E(g) \cap sE(g)s^{-1}$. As $E(g)$ is almost malnormal (Proposition~\ref{prop:dgo}), this implies that $s \in E(g)$. As $s$ has infinite order (being loxodromic in $T$), this contradicts our choice of $g$. Thus, we have proved that the intersection between the axis of $s$ and $\mathrm{Fix}(g)$ is a (possibly empty) segment of length at most $\|s\|_T$. As a consequence, $\mathrm{Fix}(g)$ and $\mathrm{Fix}(s^2gs^{-2})=s^2 \mathrm{Fix}(g)$ are disjoint, so that $g \cdot s^2gs^{-2}$ defines a $T$-loxodromic element which belongs to $H$. 

\medskip \noindent
So $H$ acts nontrivially on $T$. As $H$ is free and intersects all conjugates of subgroups in $\mathcal{P}$ trivially, the $H$-action on $T$ has cyclic arc stabilisers. If $h$ is elliptic, it follows from \cite[Corollary 5.2]{GuirardelRealTree} that one can also find a nontrivial $H$-action on a simplicial tree with cyclic edge stabilisers in which $h$ is elliptic. This contradicts the choice of $h$. Therefore, $h$ must be loxodromic in $T$, as desired. 
\end{proof}

\noindent
In some situations, being nowhere elliptic in $\calz\calp$-trees (with respect to a well-chosen family $\calp$ of subgroups) is sufficient in order to verify Condition~2 (Elementary fixator) from Proposition~\ref{prop:SimpleOne}. For instance, the strategy works for some right-angled Artin groups \cite{AutRAAG}, and also for hyperbolic and relatively hyperbolic groups thanks to the following fact. 

\begin{prop}[{Guirardel--Levitt \cite[Corollary~7.13]{GL-splittings}}]\label{prop:nowhere-elliptic-fixator}
Let $G$ be a group which is hyperbolic relative to a finite set $\calp$ of finitely generated subgroups, and let $g\in G$ be nowhere elliptic in $\calz\calp$-trees. Then $\Fix_{\Aut(G)}(g)$ contains $\langle \ad_g\rangle$ as a finite-index subgroup.
\end{prop}

\subsection{Acylindrical hyperbolicity of the automorphism group}

\noindent We start by recording the following lemma.

\begin{lemma}\label{lemma:acyl-hyp-subgroup}
Let $G$ be a group such that $\mathrm{Inn}(G)$ is infinite. If $\Aut(G)$ is acylindrically hyperbolic, then every subgroup $\mathscr{A}$ of $\Aut(G)$ that contains $\mathrm{Inn}(G)$ is acylindrically hyperbolic.
\end{lemma}

\begin{proof}
Let $\Aut(G)$ act nonelementarily and acylindrically on a hyperbolic space $X$. Since $\mathrm{Inn}(G)$ is an infinite normal subgroup of $\mathrm{Aut}(G)$, it follows from \cite[Lemma~7.2]{OsinAcyl} that the induced action of $\mathrm{Inn}(G)$ on $X$ is nonelementary. In particular the $\mathscr{A}$-action on $X$ is nonelementary, and this action is also acylindrical (being a restriction of an acylindrical action), so the desired conclusion follows. 
\end{proof}

\noindent
Given a group $G$ and a collection $\calp$ of subgroups of $G$, we denote by $\Aut(G,\calp)$ the subgroup of $\mathrm{Aut}(G)$ made of all automorphisms which send every subgroup in $\calp$ to a conjugate of a subgroup in $\mathcal{P}$. Also, recall that $G$ is \emph{one-ended relative to $\calp$} if every simplicial $H$-action on a simplicial tree with finite edge stabilisers in which all subgroups in $\calp$ are elliptic has a global fixed point. For instance, if $G$ is one-ended, then it is automatically one-ended relative to $\calp$. As an application of the theory of JSJ decompositions for relatively hyperbolic groups (we refer the reader to \cite{GL-jsj} for an account of this theory), we deduce the following statement.

\begin{prop}\label{prop:RH}
Let $G$ be a group which is hyperbolic relative to a finite collection $\calp$ of finitely generated subgroups, with no subgroup in $\calp$ equal to $G$. Assume that $G$ is not virtually free and that it is one-ended relative to $\calp$. Then $\Aut(G,\calp)$ is acylindrically hyperbolic.
\end{prop}

\begin{proof}
According to \cite[Corollary~9.20]{GL-jsj}, the group $G$ admits an $\mathrm{Aut}(G,\calp)$-invariant JSJ tree $T$ over $\mathcal{P}$-elementary subgroups. Moreover, according to \cite[Proposition~7.12]{GL-jsj}, the $G$-action on $T$ is acylindrical. Because $G$ is not virtually cyclic, either the $G$-action on $T$ has a global fixed point or it is nonelementary.

\medskip \noindent
Assume first that $G$ fixes a vertex $v$ of $T$. If $v$ is rigid, then $G$ does not split over a $\mathcal{P}$-elementary subgroup and it follows from \cite[Corollary 7.13]{GL-splittings} that $\mathrm{Out}(G,\calp)$ is finite. Consequently, $\mathrm{Inn}(G) \simeq G/Z(G)$ has finite index in $\mathrm{Aut}(G,\calp)$. As $Z(G)$ is finite, it follows from \cite{MR2507252} that $\mathrm{Aut}(G,\calp)$ is relatively hyperbolic, hence acylindrically hyperbolic as desired. If $v$ is flexible, then according to \cite[Corollary~9.20]{GL-jsj}, $G$ is $\mathcal{P}$-elementary (which is impossible), or virtually free (which is also impossible), or virtually a closed surface group. In the latter case, the acylindrical hyperbolicity of $\mathrm{Aut}(G)$ follows from \cite[Corollary~5.5]{MR4011668}. As $\mathrm{Inn}(G)$ is infinite and contained in $\Aut(G,\calp)$, it follows from Lemma~\ref{lemma:acyl-hyp-subgroup} that $\Aut(G,\calp)$ is acylindrically hyperbolic.

\medskip \noindent
From now on, assume that $G$ acts nonelementarily on $T$. By Lemma~\ref{lemma:existence-nowhere-elliptic}, there exists an element $g\in G$ which is nowhere elliptic in $\calz\calp$-trees. In particular $g$ acts loxodromically on $T$, and as the $G$-action on $T$ is acylindrical, the element $g$ is WPD in $T$. By Proposition~\ref{prop:nowhere-elliptic-fixator}, the group $\langle \mathrm{ad}_g \rangle$ has finite index in $\{ \varphi \in \mathrm{Aut}(G,\mathcal{P}) \mid \varphi(g)=g\}$.  Acylindrical hyperbolicity of $\Aut(G,\calp)$ therefore follows from Proposition~\ref{prop:JSJcanonical}.
\end{proof}

\section{Infinitely-ended groups}\label{sec:infinitely-ended}

\noindent
In this section, we prove our main theorem that automorphism groups of finitely generated infinitely-ended groups are acylindrically hyperbolic. The strategy is to apply the criterion given by Proposition \ref{prop:SimpleOne}. The key Condition~4 (Persistence of long intersections) is studied in Section~\ref{sub:BarrierSD}. The proof of our theorem is then completed in Section~\ref{sub:ThmSD}. We start by providing some extra background regarding deformation spaces of $G$-trees.

\subsection{Additional background on deformation spaces}

Let $G$ be a group. A \emph{simplicial metric $G$-tree} is a metric space $S$ obtained by equipping a simplicial $G$-tree with a $G$-invariant path metric in which every edge is made isometric to a nondegenerate compact segment. It is naturally equipped with an isometric $G$-action. Given an edge $e\subseteq S$, we denote by $\ell(e)$ the length of $e$, and by $G_e$ its stabilizer for the $G$-action. The following specific types of maps between simplicial metric $G$-trees will be used in the sequel:
\begin{itemize}
	\item A map $f:S\to T$ between simplicial metric $G$-trees is \emph{piecewise linear} if, after possibly refining the simplicial structure by adding extra valence two vertices in every edge of $S$, the map $f$ sends every \emph{edgelet} of $S$ (i.e.\ an edge for the new simplicial structure) linearly to either a vertex or an edge of $T$.  
	\item A piecewise linear $G$-equivariant map $f:S\to T$ is a \emph{collapse map} if, after possibly subdividing edges, it consists in sending one edgelet of $S$, as well as each of its $G$-translates, to either a point, or linearly to an edge of smaller length. 
	\item A piecewise linear $G$-equivariant map $f:S\to T$ is a \emph{fold} if, after possibly subdividing edges, it consists in linearly identifying a couple of edgelets $(e,e')$ based at the same vertex $v$ of $S$, by identifying the point at distance $t$ from $v$ in $e$ to the point at distance $t$ from $v$ in $e'$ for every $t\in [0,\ell(e)]$, and considering the smallest $G$-invariant equivalence relation on $S$ induced in this way. 
\end{itemize}
The following lemma is a variation over a theorem of Stallings \cite{Sta}.

\begin{lemma}\label{lemma:stallings}
Let $G$ be a finitely generated group, and let $K\ge 1$ be an integer. Let $S$ and $T$ by two minimal simplicial $G$-trees whose edge stabilisers have cardinality at most $K$, and let $f:S\to T$ be a $G$-equivariant $1$-Lipschitz piecewise linear map. 
Then $f$ is a composition of finitely many collapse maps and folds.
\end{lemma}

\begin{proof}
As $G$ is finitely generated and $S$ is minimal, there are only finitely many $G$-orbits of edges in $S$. As $f$ is $G$-equivariant and piecewise linear, we can perform a $G$-invariant subdivision of the edges of $S$, where each edge is subdivided into finitely many edgelets, and assume that $f$ sends every edge $e$ of $S$ either to a point, or linearly to an edge $f(e)$ of $T$. As $f$ is $1$-Lipschitz, we have $\ell(f(e))\le\ell(e)$, so after performing finitely many collapse maps on $S$ (at most one per $G$-orbit of edgelets), we can assume that $f$ is an isometry in restriction to every edge $e$ of $S$.

\medskip \noindent
If $S$ is a global isometry, then we are done. Otherwise, there exist two adjacent edges $e,e'\subseteq S$ such that $f(e)=f(e')$. Therefore $f$ factors through a fold of $e$ and $e'$. This operation either decreases the number of $G$-orbits of edges of $S$ (if $e$ and $e'$ were in distinct $G$-orbits), or it increases the cardinality of an edge stabilizer. The process therefore terminates after finitely many steps, thereby proving the lemma.  
\end{proof}

\noindent
Let $G$ be a finitely generated group. A \emph{deformation space of $G$-trees} (as introduced by Guirardel and Levitt in \cite{GL-deformation}) is a space $\mathscr{D}$ of equivariant isometry classes of minimal simplicial $G$-trees $T$, such that for any two trees $T,T'\in\mathscr{D}$, there exist $G$-equivariant maps from $T$ to $T'$ and from $T'$ to $T$ (equivalently all trees in $\mathscr{D}$ have the same elliptic subgroups). We equip $\mathscr{D}$ with the \emph{Gromov--Hausdorff equivariant topology}, introduced by Paulin in \cite{Pau2,Pau}. For this topology, a basis of neighborhoods of a tree $T$ is given by the sets $U_{F,K,\epsilon}$, where $F\subseteq G$ is a finite subset, where $K\subseteq T$ is a finite subset and where $\epsilon>0$, made of all trees $T'$ for which there is a finite subset $K'\subseteq T'$ and a bijection $\theta:K\to K'$ such that for all $x,y\in K$ and all $g\in F$, one has 
$$|d_{T'}(\theta(x),g\theta(y))-d_T(x,gy)|<\epsilon.$$  
We say that a deformation space $\mathscr{D}$ is \emph{$\Aut(G)$-invariant} if for every $S\in\mathscr{D}$ and every $\varphi\in\mathscr{D}$, one has $S\cdot\varphi\in\mathscr{D}$ (where $S\cdot \varphi$ the $G$-tree obtained from $S$ by precomposing the action by $\varphi$).

\subsection{Persistence of long intersections}\label{sub:BarrierSD}

\noindent
Recall the definition of nowhere elliptic elements from Definition~\ref{de:nowhere-elliptic}. The main statement of the subsection is the following proposition.

\begin{prop}\label{prop:barrier}
Let $G$ be a finitely generated infinitely-ended group, and let $T$ be a minimal simplicial $G$-tree with finite edge stabilisers. 

\noindent Then all elements of $G$ which are nowhere elliptic in $\calz$-trees have the persistence of long intersections property in $T$. 
\end{prop}

\noindent
Our proof of Proposition~\ref{prop:barrier} requires a few preliminary lemmas. First of all, we recall that the \emph{bounded backtracking constant} $\BBT(f)$ of a $G$-equivariant map $f$ from a $G$-tree $S$ to a $G$-tree $T$ is defined as the smallest $D\ge 0$ such that for every $x,y\in S$ and every $z\in [x,y]$, one has $d_T(f(z),[f(x),f(y)])\le D$. The following lemma extends \cite[Lemma~3.1]{BFH} and its proof to more general deformation spaces. The idea of using the renormalised volume $\vol$ was suggested to us by Vincent Guirardel. We denote by $\Lip(f)$ the Lipschitz constant of a map $f$.

\begin{lemma}\label{lemma:bbt}
Let $G$ be a finitely generated group and $K \geq 1$ an integer. Let $S$ and $T$ be two minimal simplicial metric $G$-trees whose edge stabilisers have cardinality at most $K$, and let $f:S\to T$ be a $G$-equivariant piecewise linear map. Then 
$$\BBT(f)\le 2K \cdot \Lip(f) \cdot \vol(S/G) $$
where $\displaystyle \vol(S/G):=\sum\limits_{e\in E(S/G)}\ell(e) / |G_e|$ (here $E(S/G)$ is finite as $G$ is finitely generated and the $G$-action on $S$ is minimal).
\end{lemma}

\begin{proof}
Up to rescaling the metric of $S$ by a factor $1/\Lip(f)$, we may assume without loss of generality that $f$ is $1$-Lipschitz. We will prove more precisely that 
$$\BBT(f)\le 2K\left(\vol(S/G)-\vol(T/G)\right).$$ 
Notice that if $U$ is a simplicial metric $G$-tree, and if $g:S\to U$ and $h:U\to T$ are $1$-Lipschitz $G$-equivariant maps, then $\BBT(g\circ h)\le\BBT(g)+\BBT(h)$. As every $G$-equivariant $1$-Lipschitz map from $S$ to $T$ factors (up to subdividing the edges of $S$) as a composition of collapses and folds (Lemma~\ref{lemma:stallings}), it is therefore enough to observe that the above inequality holds when $f$ is a collapse map or a fold. 

\medskip \noindent
First, assume that $f$ is a collapse map. Then $\BBT(f)=0$, while the right-hand side of the above inequality is always non-negative. So the inequality holds.

\medskip \noindent
Next, assume that $f$ folds two adjacent edges $e$ and $e'$ in different orbits (with the same length $\ell(e)$). On the one hand, observe that $\BBT(f)=\ell(e)$. Indeed, fix three vertices $x,y \in S$ and $z \in [x,y]$, and denote by $\sigma_1, \ldots, \sigma_r$ the subsegments of length two in $[x,y]$ which are $G$-translates of $e \cup e'$. Then $f(z)$ belongs to the segment $[f(x),f(y)]$ if $z$ is not the interior vertex of a subsegment $\sigma_i$, i.e. $d_T(f(z),[f(x),f(y)])=0$; and otherwise, if $z$ is the interior vertex of some $\sigma_i$, then the closest point of $[f(x),f(y)]$ to $f(z)$ is the image of an endpoint of $\sigma_i$, i.e. $d_T(f(z),[f(x),f(y)])=\ell(e)$. This shows that $\BBT(f)=\ell(e)$, as desired.
On the other hand, 
$$2K \left( \vol(S/G)-\vol(T/G) \right)  = 2K\ell(e)\left(\frac{1}{|G_e|}+\frac{1}{|G_{e'}|}-\frac{1}{|\langle G_e,G_{e'}\rangle|}\right).$$ Since $|\langle G_e, G_e' \rangle| \geq |\langle G_e' \rangle|$, we have
$$2K\left(\vol(S/G)-\vol(T/G)\right)  \geq \ell(e) \frac{2K}{|G_e|} \geq \ell(e)= \BBT(f)$$
as desired. 

\medskip \noindent
Finally, if $f$ folds two distinct edges $e$ and $ge$ in the same orbit, then $G_{f(e)}=\langle G_e,g\rangle$ and $\BBT(f)=\ell(e)$. Therefore 
$$2K\left(\vol(S/G)-\vol(T/G)\right)=2 \ell(e)K \left(\frac{1}{|G_e|} -\frac{1}{|\langle G_e,g\rangle|}\right).$$
As $ge \neq e$, the edge group $G_e$ is a proper subgroup of $\langle G_e,g \rangle$, hence $|\langle G_e,g \rangle | \geq 2 | G_e|$. Consequently, 
$$2K\left(\vol(S/G)-\vol(T/G)\right)\ge\ell(e)\frac{K}{|G_e|} \geq \ell(e) = \BBT(f),$$ 
as desired.
\end{proof}

\noindent 
All simplicial trees considered in the following lemma are equipped with the simplicial metric, where each edge is assigned length $1$.

\begin{lemma}\label{lemma:small-action}
Let $G$ be a finitely generated infinitely-ended group, and let $\cald$ be an $\mathrm{Aut}(G)$-invariant deformation space of minimal simplicial $G$-trees with finite edge stabilisers of bounded cardinality. Let $S\in\cald$, and let $(\varphi_n)_{n\in\mathbb{N}}\in\Aut(G)^{\mathbb{N}}$. Then $(S\cdot\varphi_n)_{n\in\mathbb{N}}$ has a subsequence that converges projectively in the equivariant Gromov--Hausdorff topology to a nontrivial real $G$-tree with virtually cyclic arc stabilisers. 
\end{lemma}

\begin{proof}
As $G$ is infinitely-ended, it contains a nonabelian free group. Therefore, by \cite{CM,Pau2}, the sequence $(S\cdot\varphi_n)_{n\in\mathbb{N}}$ has a subsequence that converges projectively to a nontrivial real $G$-tree $T$ in the Gromov--Hausdorff equvariant topology. This means that there exists a sequence $(\alpha_n)_{n\in\mathbb{N}}\in (\mathbb{R}_+^*)^{\mathbb{N}}$ such that the rescaled $G$-trees $\alpha_nS\cdot\varphi_n$ converge non-projectively to $T$. We aim to show that $T$ has virtually cyclic arc stabilisers.

\medskip\noindent If the sequence $(\alpha_n)_{n\in\mathbb{N}}$ does not converge to $0$, then there is a positive lower bound to the translation length in $\alpha_n S\cdot\varphi_n$ of every infinite-order element of $G$. So in the limit, all point stabilizers (in particular all arc stabilizers) in $T$ are finite. 

\medskip\noindent We can therefore assume that the renormalizing sequence $(\alpha_n)_{n\in\mathbb{N}}$ converges to $0$. Let $I=[x,y]\subseteq T$ be a nondegenerate arc. We aim to prove that $\Stab(I)$ is virtually cyclic, so we can assume without loss of generality that $\Stab(I)$ is infinite, as otherwise the conclusion is obvious. Let $M\ge 0$ be an upper bound to the cardinality of an edge stabiliser of a tree in $\mathscr{D}$, and let $g_0,\dots,g_M\in\Stab(I)$ be pairwise distinct elements. Choose approximations $(x_n)_{n\in\mathbb{N}}$ of $x$ and $(y_n)_{n\in\mathbb{N}}$ of $y$ in the trees $\alpha_n S\cdot\varphi_n$. Then for every $h\in\Stab(I)$, the characteristic sets of $g_0,\dots,g_M$ and $h$ in $\alpha_n S\cdot\varphi_n$ all pass arbitrary close to $x_n$ and $y_n$, and the translation lengths of $g_0,\dots,g_M$ and $h$ in $\alpha_n S\cdot\varphi_n$ are all arbitrary small compared to the distance between $x_n$ and $y_n$. As edge stabilisers of $S\cdot\varphi_n$ have cardinality at most $M$, there exists $k_n\in\{g_0,\dots,g_M\}$ which is loxodromic in $S\cdot\varphi_n$. Up to passing to a subsequence, we can assume that $k_n$ does not depend on $n$, and we denote it by $k$. As edge stabilisers of $S$ are finite, the element $k$ is WPD (with uniform constants) in all trees $S\cdot\varphi_n$. In addition, for all sufficiently large $n$, the axis of $k$ and the characteristic set of $h$ in $\alpha_n S\cdot\varphi_n$ share a subsegment of length arbitrary close to $d_T(x,y)$, while their translation length is arbitrary close to $0$. We can therefore apply Lemma~\ref{lemma:overlap} to deduce that $h\in E(k)$. Therefore $\Stab(I)\subseteq E(k)$, which is virtually cyclic.
\end{proof}

\begin{lemma}\label{lemma:barrier}
Let $G$ be a finitely generated infinitely-ended group, and let $\mathscr{D}$ be an $\mathrm{Aut}(G)$-invariant deformation space of minimal simplicial $G$-trees with finite edge stabilisers of bounded cardinality. Let $S\in\mathscr{D}$, and let $g\in G$ be an element which is nowhere elliptic in $\calz$-trees (in particular $g$ is $S$-loxodromic).

\noindent  Then $g$ has the persistence of long intersections property in $S$. 
\end{lemma}

\begin{proof}
Assume towards a contradiction that the conclusion of the lemma fails. Then we can find $C\ge 1$, a sequence of subsets $(\mathcal{X}_n)_{n\in\mathbb{N}}\in (2^G)^\mathbb{N}$ made of $S$-loxodromic elements, and a sequence $(\varphi_n)_{n\in\mathbb{N}}\in\Aut(G)^\mathbb{N}$, such that for every $n\in\mathbb{N}$, the intersection $$\Axis_S(g)\cap\bigcap_{h\in\mathcal{X}_n}\Axis_S(h)$$ contains $n$ fundamental domains of the axis of $g$, while either $\varphi_n(\mathcal{X}_n)$ contains an $S$-elliptic element, or the intersection $$\Axis_S(\varphi_n(g))\cap\bigcap_{h\in\mathcal{X}_n}\Axis_S(\varphi_n(h))$$ contains at most $C$ fundamental domains of the axis of $\varphi_n(g)$. 

\medskip \noindent
By Lemma~\ref{lemma:small-action}, up to passing to a subsequence, the sequence $(S\cdot\varphi_n)_{n\in\mathbb{N}}$ converges projectively to a nontrivial $G$-tree $T$ with virtually cyclic arc stabilisers.  
This means that there exists a sequence $(\alpha_n)_{n\in\mathbb{N}}\in (\mathbb{R}_+^*)^\mathbb{N}$ such that the rescaled $G$-trees $\alpha_n S\cdot\varphi_n$ converge (non-projectively) to $T$. For every $n\in\mathbb{N}$, we let $S_n:=\alpha_n S\cdot\varphi_n$.

\medskip\noindent We will now construct a piecewise linear Lipschitz $G$-equivariant map $f:S\to T$. Indeed, as $G$ is finitely generated and $S$ is a minimal $G$-tree, the tree $S$ has finitely many $G$-orbits of vertices, and finitely many $G$-orbits of edges. Let $V$ be a finite set of representatives of the $G$-orbits of vertices of $S$. For every $v\in V$, the $G$-stabiliser $G_{v}$ of $v$ is elliptic in all trees in $\mathscr{D}$, and therefore it is also elliptic in the limiting tree $T$, so we can find $x_v\in T$ which is fixed by $G_{v}$. Now, every edge $e\subseteq S$ is of the form $[g_1v_1,g_2v_2]$ with $g_1,g_2\in G$ and $v_1,v_2\in V$. The $G$-stabiliser of $e$ fixes both $g_1v_1$ and $g_2v_2$, and therefore it also fixes $g_1x_{v_1}$ and $g_2x_{v_2}$, as well as the whole segment joining them. We can therefore define a $G$-equivariant map $f:S\to T$ that linearly sends the edge $[g_1v_1,g_2v_2]$ to the segment $[g_1x_{v_1},g_2x_{v_2}]$. As there are only finitely many $G$-orbits of edges in $S$, this map is Lipschitz, thus proving our claim.   

\medskip\noindent Let $L>0$ be such that the map $f$ is $L$-Lipschitz. We now claim that for every sufficiently large $n\in\mathbb{N}$, there exists a $2L$-Lipschitz piecewise linear $G$-equivariant map $f_n:S\to S_n$. Indeed, this is proved by approximating every point $x_v$ by a sequence of vertices $w_n(v)\in S_n$, in such a way that the $G$-stabilizer of $v$ stabilizes $w_n(v)$, and considering a piecewise linear map $f_n$ sending $v$ to $w_n(v)$ (as constructed in the previous paragraph). In this way $f_n$ converges to $f$ in the Gromov--Hausdorff equivariant topology on the set of morphisms considered by Guirardel and Levitt in \cite[Section~3.2]{GL-os}, and as $f$ is $L$-Lipschitz, it follows that $f_n$ is $2L$-Lipschitz for $n$ sufficiently large. 

\medskip\noindent Up to restricting to a subsequence, we will now assume that for every $n\in\mathbb{N}$, there exists a $2L$-Lipschitz piecewise linear $G$-equivariant map $f_n:S\to S_n$. Let $K>0$ be the maximal cardinality of an edge stabiliser of a tree in $\mathscr{D}$ (which is bounded by assumption). By Lemma~\ref{lemma:bbt}, letting $M:=4KL\cdot\vol(S/G)$, we have $\BBT(f_n)\le M$ for every $n\in\mathbb{N}$.

\medskip \noindent
Fix $n \in \mathbb{N}$. Let $I_n$ be a segment contained in $$\Axis_S(g)\cap\bigcap_{h\in\mathcal{X}_n}\Axis_S(h)$$ which  contains $n$ fundamental domains of $\Axis_S(g)$. As $\BBT(f_n)\le M$, the axis in $S$ of every $S$-loxodromic element $k\in G$ is sent by $f_n$ within the $M$-neighbourhood of $\Char_{S_n}(k)$: this follows from the fact that for every $x\in S$, one has $$d_{S_n}(f_n(kx),[f_n(x),f_n(k^2x)])=d_{S_n}(f_n(x),\Char_{S_n}(k)).$$ Therefore $f_n(I_n)$ is contained both in an $M$-neighbourhood of $\Axis_{S_n}(g)$, and also in an $M$-neighbourhood of $\Char_{S_n}(h)$ for every $h\in\mathcal{X}_n$. In addition $f_n(I_n)$ contains $n$ fundamental domains of $\Axis_{S_n}(g)$: indeed $I_n$ contains a subsegment of the form $[x,g^nx]$, and denoting by $y$ the projection of $f_n(x)$ to $\Axis_{S_n}(g)$, the image $f_n([x,g^n x])$ contains $[y,g^ny]$. 

\medskip \noindent Assume first that all elements in $\mathcal{X}_n$ are $S_n$-loxodromic. If $$\Axis_{S_n}(g)\cap\bigcap_{h\in\mathcal{X}_n}\Axis_{S_n}(h)$$ is empty, then $f_n(I_n)$ has diameter at most $2M$, so $n\|g\|_{S_n}\leq 2M.$ If  $$\Axis_{S_n}(g)\cap\bigcap_{h\in\mathcal{X}_n}\Axis_{S_n}(h)$$ is non-empty, since it contains at most $C$ fundamental domains of $\Axis_{S_n}(g)$, we deduce that 
$$n \|g\|_{S_n} \leq C \| g\|_{S_n} +2M, \text{ hence } ||g||_{S_n}\le\frac{2M}{n-C}.$$ 
In both cases, we deduce in the limit that $||g||_T=0$. This contradicts the fact that $g$ is nowhere elliptic in $\calz$-trees.

\medskip \noindent Assume now that there exists an element $h_n\in\mathcal{X}_n$ which is $S_n$-elliptic. As $h_n$ is $S$-loxodromic, it is of infinite order, and therefore $h_n$ fixes exactly one point in $S_n$. We deduce that $n\|g\|_{S_n}\le 2M$ and reach a contradiction as in the previous case.
\end{proof}

\begin{lemma}\label{lemma:collapse-fundamental-domains}
Let $G$ be a group. Let $S$ and $T$ be two simplicial $G$-trees, and assume that there exists a $G$-equivariant collapse map $S\to T$. Let $g,h\in G$, and assume that $g$ is loxodromic in both $S$ and $T$. Let $n_S$ (resp.\ $n_T$) be the number of fundamental domains of $\Axis_S(g)$ (resp.\ $\Axis_T(g)$) contained in $\Char_S(h)$ (resp.\ $\Char_T(h)$).

\noindent Then $n_S\le n_T\le n_S+2$.
\end{lemma}

\begin{proof}
Every fundamental domain of $\Axis_S(g)$ is sent to a fundamental domain of $\Axis_T(g)$ under the collapse map $f:S\to T$, and $\Char_S(h)$ is sent to $\Char_T(h)$; this shows that $n_S\le n_T$. Conversely, let $I_1,\dots,I_k$ be fundamental domains of $\Axis_T(g)$ that are all contained in $\Char_T(h)$, aligned in this order along $\Axis_T(g)$, and pairwise intersecting in at most one point. Then we can lift $I_1,\dots,I_k$ to fundamental domains $\tilde{I}_1,\dots,\tilde{I}_k$ of $\Axis_S(g)$, aligned in this order, pairwise intersecting in at most one point, and the fact that all $I_i$ are contained in $\Char_T(h)$ implies that all $\tilde{I}_i$ intersect $\Char_S(h)$ nontrivially. By convexity of $\Char_S(h)$, this implies that $\tilde{I}_2,\dots,\tilde{I}_{k-1}$ are contained in $\Char_S(h)$. This shows that $n_T\le n_S+2$.    
\end{proof}

\begin{lemma}\label{lemma:barrier-2}
Let $G$ be a finitely generated infinitely-ended group, let $T$ be a simplicial $G$-tree. Let $g\in G$ be an element which is nowhere elliptic in $\calz$-trees. Assume that there exist an $\Aut(G)$-invariant deformation space $\cald$ of simplicial $G$-trees with finite edge stabilisers of bounded cardinality, and a tree $S\in\cald$ coming with a $G$-equivariant collapse map $S\to T$.

\noindent 
Then $T$ has finite edge stabilisers, and $g$ is $T$-loxodromic and it has the persistence of long intersections property in $T$.
\end{lemma}

\begin{proof}
Notice that preimage in $S$ of a vertex of $T$ is a subtree. Consequently, an edge stabiliser in $T$ stabilises two disjoint subtrees in $S$, and so it must be finite because it stabilises the bridge between these two subtrees and because edge stabilisers in $S$ are finite. Therefore, $T$ has finite edge stabilisers. It follows that $g$ is $T$-loxodromic since $g$ is nowhere elliptic in $\calz$-trees. 

\medskip \noindent
Let $C\ge 1$, and let $n(C)$ be the number coming from the persistence of long intersections property for $g$ in $S$ (ensured by Lemma~\ref{lemma:barrier}). Let $\varphi\in \Aut(G)$, and let $\mathcal{X}\subseteq G$. Assume that all elements in $\mathcal{X}$ are $T$-loxodromic and that $$\Axis_T(g)\cap\bigcap_{h\in\mathcal{X}}\Axis_T(h)$$ contains $n(C)+2$ fundamental domains of the axis of $g$. Using Lemma~\ref{lemma:collapse-fundamental-domains}, we deduce that $$\Axis_S(g)\cap\bigcap_{h\in\mathcal{X}}\Axis_S(h)$$ contains $n(C)$ fundamental domains of the axis of $g$. By Lemma~\ref{lemma:barrier}, all elements in $\varphi(\mathcal{X})$ are $S$-loxodromic, and $$\Axis_S(\varphi(g))\cap\bigcap_{h\in\mathcal{X}}\Axis_S(\varphi(h))$$ contains $C$ fundamental domains of the axis of $\varphi(g)$. Using Lemma~\ref{lemma:collapse-fundamental-domains} again, this implies that the same is true in $T$. 
\end{proof}

\noindent
We are now in position to complete our proof of the main result of the section.

\begin{proof}[Proof of Proposition \ref{prop:barrier}.]
Since $G$ is finitely generated and $T$ is a minimal $G$-tree, there are only finitely many $G$-orbits of edges. We can therefore let $K$ be the maximal cardinality of an edge stabiliser in $T$, which is finite. Let $\mathscr{A}_K$ be the collection of all finite subgroups of $G$ of cardinality at most $K$, and let $\mathscr{D}$ be the JSJ deformation space of $G$ over subgroups in $\mathscr{A}_K$: this exists by Linnell's accessibility \cite{Lin}, see e.g.\ \cite[Remark~3.3]{GL-jsj}, and it is $\Aut(G)$-invariant. In addition, the tree $T$ is a splitting of $G$ over $\mathscr{A}_K$, and as such it is universally elliptic (i.e.\ its edge stabilisers, being finite, are elliptic in all splittings of $G$ over $\mathscr{A}_K$). It thus follows from \cite[Lemma~2.15]{GL-jsj} that $T$ has a refinement in $\mathscr{D}$. The conclusion therefore follows from Lemma~\ref{lemma:barrier-2}. 
\end{proof}

\subsection{Acylindrical hyperbolicity of the automorphism group}\label{sub:ThmSD}

\noindent
We are now ready to prove the main result of the paper.

\begin{theo}\label{thm:AutEnds}
Let $G$ be a finitely generated infinitely-ended group. Then $\Aut(G)$ is acylindrically hyperbolic.
\end{theo}

\begin{proof}
Our goal is to apply the criterion provided by Proposition \ref{prop:SimpleOne}.

\medskip \noindent
Let $S$ be a nontrivial minimal simplicial $G$-tree with finite edge stabilisers (of bounded cardinality), and let $\calp$ be a finite set of representatives of the vertex stabilisers of $S$. It follows from Bowditch's definition of relative hyperbolicity \cite[Definition 2]{Bow} that the group $G$ is hyperbolic relative to $\calp$. Let $g\in G$ be an element which is nowhere elliptic in $\calz\calp$-trees, as given by Lemma \ref{lemma:existence-nowhere-elliptic}. By Proposition~\ref{prop:nowhere-elliptic-fixator}, the group $\Fix_{\Aut(G)}(g)$ contains $\langle \ad_g\rangle$ as a finite-index subgroup. Also, notice that $g$ must be $S$-loxodromic, so it is WPD in $S$ as edge stabilisers are finite. Therefore, Lemma~\ref{lemma:elementary-fixator} implies that $\Stab_{\Aut(G)}(E(g))$ contains $\langle\ad_g\rangle$ as a finite-index subgroup.

\medskip \noindent 
Let $H:=\Stab_{\Aut(G)}(E(g))$, and let $\tilde{H}:=\langle H,\mathrm{Inn}(G)\rangle$. Then $\tilde{H}$ has finite image in $\Out(G)$, and therefore it contains $\mathrm{Inn}(G)$ as a finite index subgroup. As $G$ is infinitely-ended, its center $Z(G)$ is finite (as a consequence, for instance, of Lemma \ref{lem:center}), so $\mathrm{Inn}(G)\simeq G/Z(G)$ is again finitely generated and infinitely-ended. Therefore $\tilde{H}$ is finitely generated and infinitely-ended. By Stallings' theorem \cite{Sta,Sta2}, there exists a nontrivial simplicial $\tilde{H}$-tree $T$ with finite edge stabilisers. Up to replacing $T$ with an $\tilde{H}$-invariant subtree, we shall assume that $T$ is minimal. In particular $T$ can be viewed as a $G$-tree, and as such it satisfies Assumption~3 from Proposition~\ref{prop:SimpleOne} (Nielsen realisation). Notice that the element $g$ has to be $T$-loxodromic, so it is WPD for the $G$-action on $T$ as edge stabilisers are finite. As $G$ is not virtually cyclic, the $G$-action on $T$ is nonelementary.

\medskip\noindent We now apply Proposition~\ref{prop:SimpleOne} to the element $g$ and the tree $T$. As edge stabilisers in $T$ have bounded cardinality, all elements of $G$ which are $T$-loxodromic are uniformly WPD, showing that Assumption~1 holds. Assumptions~2 (Elementary fixator) and~3 (Nielsen realisation) have been checked above. Notice that every element $g'$ in the $\Aut(G)$-orbit of $g$ is nowhere elliptic in $\calz$-trees: indeed, if $g'=\varphi(g)$ were elliptic in a tree $T$ with virtually cyclic edge stabilizers, then $g$ would be elliptic in $T\cdot\varphi^{-1}$ which also has virtually cyclic edge stabilizers. Thus Assumption~4 (Persistence of long intersections) follows from Proposition~\ref{prop:barrier}. Proposition~\ref{prop:SimpleOne} therefore implies that $\Aut(G)$ is acylindrically hyperbolic.
\end{proof}

\begin{remark}\label{rk:finite-generation}
The finite generation assumption on $G$ is crucial in Theorem~\ref{thm:AutEnds}. Here is an example of a (non-finitely generated) group $G$ that splits as a free product, for which $\Aut(G)$ fails to be acylindrically hyperbolic. As pointed by the referee, the automorphism group of a free group of infinite countable rank is not acylindrically hyperbolic (which can be proved by following the lines of the argument below). Here is another example where $\Aut(G)$ fails to be acylindrically hyperbolic, even though $G$ splits as a free product of finitely many subgroups that are either freely indecomposable or isomorphic to $\mathbb{Z}$. 

\medskip \noindent 
Let $Z$ be the direct sum of countably many copies of $\mathbb{Z}$, and let $G=Z\ast Z$. We claim that $\Aut(G)$ is not acylindrically hyperbolic. Indeed, assume towards a contradiction that $\Aut(G)$ acts acylindrically and nonelementarily on a hyperbolic space $X$. Then $\mathrm{Inn}(G)$ acts nonelementarily on $X$, see e.g.\ \cite[Lemma~7.2]{OsinAcyl}. Thus, some inner automorphism is WPD with respect to the $\Aut(G)$-action on $X$, which implies in particular that some inner automorphism has virtually cyclic centralizer. 

\medskip \noindent But on the other hand, we claim that every inner automorphism $\ad_g$ has non-virtually cyclic centralizer, which leads to a contradiction. Indeed, every $g\in G$ is contained in a subgroup of the form $Z_n\ast Z_n$, where $Z_n\subseteq Z$ is the direct sum of the first $n$ copies of $\mathbb{Z}$, and therefore $g$ is fixed by every automorphism $\varphi$ of $G$ that preserves each of the two copies of $Z$ from the free product $G=Z\ast Z$ and fixes $Z_n$ in each factor. This implies that the inner automorphism $\ad_g$ commutes with all such $\varphi$. 
\end{remark}

\begin{remark}
Combined with Proposition~\ref{prop:SimpleOne}, our proof of Theorem~\ref{thm:AutEnds} also shows that for every element $g\in G$ which is nowhere elliptic in $\calz\calp$-trees, the inner automorphism $\ad_g$ is a generalised loxodromic element of $\Aut(G)$. 
\end{remark}

\section{Relatively hyperbolic groups: the general case}\label{sec:rh}

\noindent
This last section is dedicated to the proof of Theorem \ref{thm:IntroRHacyl} from the introduction. 

\begin{thm}\label{thm:BigRH}
Let $G$ be a group which is hyperbolic relative to a finite collection $\mathcal{P}$ of finitely generated proper subgroups. Assume that $G$ is not virtually cyclic. Then $\mathrm{Aut}(G, \mathcal{P})$ is acylindrically hyperbolic.
\end{thm}

\begin{proof}
If $G$ is infinitely-ended, then we know from Theorem \ref{thm:AutEnds} that $\mathrm{Aut}(G)$ is acylindrically hyperbolic. As $\mathrm{Inn}(G)$ is infinite and contained in $\Aut(G,\calp)$, it then follows from Lemma~\ref{lemma:acyl-hyp-subgroup} that $\Aut(G,\calp)$ is acylindrically hyperbolic.

\medskip \noindent
Suppose now that $G$ is one-ended. In particular, $G$ is not virtually free. The acylindrical hyperbolicity of $\mathrm{Aut}(G, \mathcal{P})$ then follows from Proposition~\ref{prop:RH}. 
\end{proof}

\begin{cor}
Let $G$ be a group which is hyperbolic relative to a finite collection $\mathcal{P}$ of finitely generated proper subgroups. Assume that $G$ is not virtually cyclic, and that the groups in $\mathcal{P}$ are not relatively hyperbolic. Then $\mathrm{Aut}(G)$ is acylindrically hyperbolic.
\end{cor}

\begin{proof}
Let $\mathcal{P}'$ denote the collection of subgroups obtained from $\mathcal{P}$ by removing all the finite groups. Notice that $G$ is also hyperbolic relative to $\mathcal{P}'$. By \cite[Lemma~3.2]{MO2} (a reference for which we thank the referee), we have $\mathrm{Aut}(G)= \mathrm{Aut}(G, \mathcal{P}')$. Using Theorem~\ref{thm:BigRH}, it follows that $\Aut(G)$ is acylindrically hyperbolic, as desired.
\end{proof}

\addcontentsline{toc}{section}{References}

\bibliographystyle{alpha}
{\footnotesize\bibliography{AutEnded}}

\footnotesize
\begin{flushleft}
Anthony Genevois\\
Institut Montpellierain Alexander Grothendieck, 499-554 Rue du Truel, 34090 Montpellier, France \\
\emph{e-mail:}\texttt{anthony.genevois@umontpellier.fr}\\[5mm]

Camille Horbez\\
Universit\'e Paris-Saclay, CNRS,  Laboratoire de math\'ematiques d'Orsay, 91405, Orsay, France \\
\emph{e-mail:}\texttt{camille.horbez@universite-paris-saclay.fr}\\[8mm]
\end{flushleft}

\end{document}